\documentclass[a4paper,12pt,oneside]{article}
\usepackage[utf8]{inputenc}
\usepackage[english]{babel}
\usepackage{amsmath}
\usepackage{amsthm}
\usepackage{amssymb}
\usepackage{mathtools}
\usepackage{tikz}
\usepackage{verbatim}
\usepackage{fancyvrb}
\usepackage{array}
\usepackage{adjustbox}
\usepackage{changepage}
\usepackage{subcaption}
\usepackage{breqn}
\usepackage{enumerate}
\usepackage{wrapfig}
\usepackage{bbm}
\usepackage{bm}
\usepackage{enumitem}
\usepackage{csquotes}
\usepackage{hyperref}
\usepackage{graphicx}
\theoremstyle{definition}
\newtheorem{defi}{Definition}[section] 
\newtheorem{thm}[defi]{Theorem} 
\newtheorem{lem}[defi]{Lemma}

\newtheorem{prop}[defi]{Proposition}

\newtheorem{cor}[defi]{Corollary}

\newcommand{\N}{\mathbb{N}}
\newcommand{\Z}{\mathbb{Z}}

\newcommand{\C}{\mathbb{C}}
\newcommand{\R}{\mathbb{R}}

\newcommand{\BO}{\mathcal{O}}
\title{Multiple orthogonal polynomials associated with the exponential integral}
\author{Walter Van Assche\footnote{Department of Mathematics, KU Leuven, Celestijnenlaan 200B box 2400, BE-3001 Leuven, Belgium. Supported by FWO project G.0C9819N.
$\{$thomas.wolfs,walter.vanassche$\}$@kuleuven.be}  \and Thomas Wolfs\footnotemark[1]}

\begin{document}
\date{\today}
\maketitle

\begin{abstract}
We introduce a new family of multiple orthogonal polynomials satisfying orthogonality conditions with respect to two weights $(w_1,w_2)$ on the positive real line, with $w_1(x)=x^\alpha e^{-x}$ the gamma density and $w_2(x) = x^\alpha E_{\nu+1}(x)$ a density related to the exponential integral $E_{\nu+1}$. We give explicit formulas for the type I functions and type II polynomials, their Mellin transform, Rodrigues formulas, hypergeometric series and recurrence relations. We determine the asymptotic distribution of the (scaled) zeros of the type II multiple orthogonal polynomials and make a connection to random matrix theory. Finally, we also consider a related family of mixed type multiple orthogonal polynomials.
\medbreak


\end{abstract}

\section{Introduction}

Multiple orthogonal polynomials extend the notion of (regular) orthogonal polynomials by distributing orthogonality conditions over several measures instead of just one measure. Each multiple orthogonal polynomial corresponds to a multi-index $\vec{n}=(n_1,\dots,n_r)$ in $\Z_{\geq 0}^r$ that determines the way in which the associated orthogonality conditions are distributed. There are two main types of multiple orthogonal polynomials, see for example \cite[Chapter~23]{Ismail} and \cite[Chapter 4]{NikiSor} for an introduction.
\medbreak

Let $\mu_1,\dots,\mu_r$ be $r$ positive Borel measures, with finite moments, supported on the real line. We start by introducing the type I multiple orthogonal polynomials. These are vectors of polynomials $(A_{\vec{n},1},\dots,A_{\vec{n},r})$, where the degree of $A_{\vec{n},j}$ is at most $n_j-1$, that satisfy the orthogonality conditions
	$$\sum_{j=1}^r \int_{\R} A_{\vec{n},j}(x) x^k d\mu_j(x)=0,\quad k=0,\dots,\left|\vec{n}\right|-2.$$
The orthogonality conditions result in a system of $\left|\vec{n}\right|-1$ homogeneous linear equations for the $\left|\vec{n}\right|$ coefficients of the $A_{\vec{n},j}$. The existence of non-zero vectors of polynomials is therefore guaranteed. Another way to set up the orthogonality conditions is via the type II multiple orthogonal polynomials. These are polynomials $P_{\vec{n}}$ of degree at most $\left|\vec{n}\right|=n_1+\dots n_r$ that satisfy the orthogonality conditions
		$$\int_{\R} P_{\vec{n}}(x) x^k d\mu_j(x)=0,\quad k=0,\dots,n_j-1, \ j=1,\dots,r.$$
The orthogonality conditions result in a system of $\left|\vec{n}\right|$ homogeneous linear equations for the $\left|\vec{n}\right|+1$ coefficients of $P_{\vec{n}}$. The existence of non-zero polynomials is therefore again guaranteed. 
It is also possible to combine type I and type II orthogonality conditions in order to obtain so-called mixed type multiple orthogonal polynomials. They appear in the work of Sorokin \cite{Sor_mixed} and in the study of non-intersecting Brownian motions with several starting points and endpoints \cite{DaemsKuijl}. Later, they were also studied in \cite{AF-FP-M,FP-LG-LL-S,F-MP-MC}. Other applications of multiple orthogonal polynomials can be found in analytic number theory (Diophantine approximation), for example in Apéry's proof \cite{Apery} of the irrationality of $\zeta(3)$ as explained by Beukers in \cite{Beukers}, or in random matrix theory \cite{Kuijlaars1,Kuijlaars2}.
\medbreak

In \cite{Mahler}, Mahler introduced the notion of a perfect system. A system of measures is said to be perfect if, for all multi-indices, every type I and type II polynomial attains its maximal degree. In that case, the multiple orthogonal polynomials are uniquely determined by the orthogonality conditions up to a non-zero scalar multiplication. The system of equations then has a unique solution if we add a normalization condition. For the type I polynomials, we can impose that additionally
    $$ \sum_{j=1}^r \int_{\R} A_{\vec{n},j}(x) x^{\left|\vec{n}\right|-1} d\mu_j(x)=1. $$
The (unique) type I polynomial $(A_{\vec{n},1},\dots,A_{\vec{n},r})$ for which this holds will be denoted by $(A_{\vec{n},1}^{(I)},\dots,A_{\vec{n},r}^{(I)})$. We can normalize the type II polynomials by demanding that they are monic. The (unique) monic type II polynomial $P_{\vec{n}}$ will be denoted by $P_{\vec{n}}^{(II)}$.

Most of the well-known examples of perfect systems arise via the so-called Angelesco systems and AT-systems. Angelesco systems were introduced in \cite{Ang} and can be produced by considering 
measures supported on intervals that may touch at the endpoints but are otherwise disjoint. For an AT-system, one considers multiple measures supported on the same interval that are sufficiently different. Particular interest goes to the so-called Nikishin systems, see \cite[Chapter 5, \S 7]{NikiSor}, where this is called an MT-system, and \cite{FidLop1,FidLop2} where it is proven that Nikishin systems are perfect.
\medbreak

We will investigate a family of multiple orthogonal polynomials that will be an extension of the Laguerre polynomials, a family of classical orthogonal polynomials. Laguerre polynomials arise by considering orthogonality conditions with respect to the gamma density $x^\alpha e^{-x}$, where $\alpha>-1$, on $[0,\infty)$, see, e.g., \cite[Chapter 5]{Szego} for some properties. Several extensions of these polynomials have been studied before. For example, one can form an Angelesco system by considering the weights $x^\alpha e^{-\left|x\right|}$ or $x^\alpha e^{-x^r}$, where $\alpha>-1$, on $r$ intervals that form an $r$-star, see \cite{Sor} and \cite{LeursVA} respectively. Known extensions to AT-systems are $(x^\alpha e^{-x},x^\beta e^{-x})$, where $\alpha,\beta>-1$ and $\alpha-\beta\not\in\Z$, and $(x^\alpha e^{-c_1 x},x^\alpha e^{-c_2 x})$, where $\alpha>-1$ and $c_1,c_2>0$ are distinct, both on $[0,\infty)$. The resulting multiple orthogonal polynomials are then called multiple Laguerre polynomials of the first and second kind and were introduced in \cite{Sor} and \cite{NikiSor}. Their properties were investigated in \cite{CoussVA} and \cite{AptBranqVA}. The type I and type II multiple orthogonal polynomials in this article (Section 2 and 3 respectively) can be seen as multiple Laguerre polynomials of the third kind.
\medbreak

Observe that for the multiple Laguerre polynomials of the first and second kind the additional weight function is still a gamma density, but with different parameters. We propose to add a different kind of orthogonality weight as the second weight: one that is related to the exponential integral. The (generalized) exponential integral is defined on $\R_{>0}$ as
    $$ E_{\nu}(x) = \int_1^\infty \frac{e^{-xt}}{t^{\nu}} dt, $$
for every $\nu\in\R$. This function is connected to the incomplete gamma function, given by 
$$\Gamma_{\nu}(x) = \int_x^\infty t^{\nu-1} e^{-t} dt,$$
via the relation $ E_{\nu}(x) = x^{\nu-1} \Gamma_{-\nu+1}(x) $. Orthogonal polynomials with respect to the exponential integral have been investigated before via a numerical approach in \cite{Gautschi}. We believe that it is more natural to consider multiple orthogonality and include a related gamma density as a weight in the system as well.
\medbreak


We will consider a system of two weights $(w_1,w_2)$ on $[0,\infty)$ with
    $$ w_1(x)=x^{\alpha} e^{-x},\quad w_2(x)=x^{\alpha}E_{\nu+1}(x) .$$
In order for all the moments of $w_1$ and $w_2$ to exist, we will assume that $\alpha>-1$ and $\alpha+\nu > -1$. The moments are then given by
\begin{equation*}
    \int_0^\infty x^k w_1(x) dx = \Gamma(k+\alpha+1), \quad \int_0^\infty x^k w_2(x) dx = \frac{\Gamma(k+\alpha+1)}{k+\alpha+\nu+1},
\end{equation*}
where 
$$\Gamma(s)=\int_0^\infty x^{s-1} e^{-x} dx$$ 
denotes the gamma function. We will also exclude $\nu\in\Z_{\leq -1}$ because for these values the weights will not be linearly independent over the polynomials and the system won't be perfect.
\medbreak

The described system is perfect if $\nu>-1$, because it can be written as a Nikishin system; the ratio $w_2/w_1$ is given by the Stieltjes transform of a gamma density on the negative real line. Indeed, if $\nu>-1$, we can use \cite[Eq. 3.383.10]{GR} to write
$$\frac{w_2(x)}{w_1(x)} = \frac{1}{\Gamma(\nu+1)} \int_{-\infty}^0 \frac{(-t)^{\nu} e^{t}}{x-t} dt.$$
For $\nu=-n+v\leq -1$ with $n\in\N$ and $-1<v\leq 0$, the system is close to a Nikishin system in the sense that the ratio $w_2/w_1$ will be equal to $x^{-n}$ times the Stieltjes transform of a gamma density on $(-\infty,0]$, plus a rational function with a pole of order $n$ at $0$. To show this, we can apply \cite[Eq. 8.356.5]{GR} to obtain
$$ \frac{w_2(x)}{w_1(x)} = \frac{(-v)_n}{x^n} e^{x} E_{v+1}(x) + \frac{1}{x^n} \sum_{j=0}^{n-1} \frac{x^j}{\Gamma(j-v+1)}, $$
where $(z)_n=z(z+1)\dots(z+n-1)$ denotes the Pochhammer symbol, and we can use the result in the previous case. Note that if $\nu\in\Z_{\leq -1}$ (so $v=0$) the Stieltjes transform disappears and only the rational function remains, which means that the system $(w_1,w_2)$ is not perfect as stated before.

Another property of the system of weights is that it satisfies the matrix Pearson equation
$$ \left[  x \begin{pmatrix} w_1 \\ w_2  \end{pmatrix} \right]' 
= \begin{pmatrix} \alpha+1-x & 0 \\ -1 & \alpha+\nu+1 \end{pmatrix} 
\begin{pmatrix}  w_1 \\ w_2  \end{pmatrix} . $$
\medbreak

In Section 4, we will extend the results about the type I and type II multiple orthogonal polynomials to certain mixed type multiple orthogonal polynomials. First, we will consider mixed type orthogonality with respect to the two pairs of weights $(e^{-x},E_{\nu+1}(x))$ and $(x^{\alpha},x^{\beta})$ on $[0,\infty)$. The resulting mixed type multiple orthogonal polynomials are of interest because they resemble the rational approximants used by Apéry in his proof of irrationality of $\zeta(3)$. We believe this new family of mixed type multiple orthogonal polynomials may be relevant in the rational approximation of Euler's constant $\gamma$. Afterwards, we will investigate the dual setting in which we will reverse the order of the weights. In both settings, we will assume that $\alpha,\beta,\nu,\alpha+\nu,\beta+\nu>-1$ for the same reasons as before. Additionally, we will exclude $\alpha-\beta\in\Z$ in order for the pair of weights $(x^{\alpha},x^{\beta})$ to be linearly independent over the polynomials. It is known that in that case the pair of weights forms an AT-system on $(0,\infty)$. In fact, under the additional restriction that $0<\alpha-\beta<1$, it is even a Nikishin system because of the relation (see \cite[Eq. 5.12.3]{DLMF})
    $$x^{\beta-\alpha} = \frac{1}{\Gamma(1-\alpha+\beta)\Gamma(\alpha-\beta)} \int_{-\infty}^0 \frac{(-t)^{\beta-\alpha}}{x-t} dt.$$

\section{Type I multiple orthogonal polynomials} \label{STI}
A type I multiple orthogonal polynomial for the two weight functions $(x^\alpha e^{-x},x^\alpha E_{\nu+1}(x))$ that corresponds to the multi-index $(n,m)$ is a vector of polynomials $(A_{n,m},B_{n,m})$ with $\deg A_{n,m} = n-1$ and $\deg B_{n,m} = m-1$ for which the type I function 
\begin{equation}  \label{defF}
   F_{n,m}(x) = A_{n,m}(x) e^{-x} + B_{n,m}(x) E_{\nu+1}(x)
\end{equation}  
satisfies the orthogonality conditions
\begin{equation*}
   \int_0^\infty F_{n,m}(x) x^{\alpha+k} dx = 0, \quad k=0,\dots, n+m-2.  
\end{equation*}   
Such polynomials (having maximal degree) exist for every multi-index because of the perfectness of the system and they are unique up to a non-zero scalar multiplication. We may normalize the type I function by considering the normalization condition
\begin{equation}  \label{I_norm}
   \int_0^\infty F_{n,m}(x) x^{\alpha+n+m-1} dx = 1.  
\end{equation}
The type I function for which this holds will be denoted by $F_{n,m}^{(I)}$.

\subsection{Convolution and Rodrigues formula} \label{I_S1}

We will be able to identify the type I function, and hence the underlying type I multiple orthogonal polynomials, for multi-indices with $n+1\geq m$ by careful analysis of its Mellin transform. The Mellin transform of a locally Lebesgue integrable function $f:(0,\infty)\to \C$ is defined by the relation
    $$ \hat{f}(s) = \int_0^\infty f(x) x^{s-1} dx, $$
whenever the integral exists.

\begin{lem} \label{I_Mellin_lem}
Suppose that $n+1\geq m$. The Mellin transform of the normalized type I function $F_{n,m}^{(I)}$ is given by $\mathcal{F}_{n,m} \cdot \hat{F}_{n,m}(s)$, where
\begin{equation} \label{I_Mellin}
    \hat{F}_{n,m}(s) = \frac{\Gamma(s)}{(s+\nu)_m} (\alpha+1-s)_{n+m-1},\quad s>\max\{0,-\nu\},
\end{equation}
and
\begin{equation} \label{I_norm_ct}
    \mathcal{F}_{n,m}  =  (-1)^{n+m-1} \frac{(\alpha+\nu+1)_{n+2m-1}}{(n+m-1)! (\alpha+1)_{n+m-1} (\alpha+\nu+1)_{n+m-1} \Gamma(\alpha+1)}.
\end{equation}
\end{lem}
\begin{proof}
Take a type I function $F_{n,m}$ and expand the polynomials $A_{n,m}(x)$ and $B_{n,m}(x)$ as
    $$ A_{n,m}(x) = \sum_{k=0}^{n-1} a_k x^k,\quad B_{n,m}(x) = \sum_{l=0}^{m-1} b_l x^l. $$
The Mellin transform $\hat{F}_{n,m}(s)$ of $F_{n,m}(x)=A_{n,m}(x) e^{-x} + B_{n,m}(x) E_{\nu+1}(x)$ then becomes 
    $$ \hat{F}_{n,m}(s) = \sum_{k=0}^{n-1} a_k \int_0^\infty e^{-x} x^{k+s-1} dx + \sum_{l=0}^{m-1} b_l \int_0^\infty E_{\nu+1}(x) x^{l+s-1} dx, $$
which may be written as 
    $$ \sum_{k=0}^{n-1} a_k \Gamma(k+s) + \sum_{l=0}^{m-1} b_l \frac{\Gamma(l+s)}{l+s+\nu}. $$
We can take a common factor out of the sums and get
    $$ \Gamma(s) \left( \sum_{k=0}^{n-1} a_k (s)_k + \sum_{l=0}^{m-1} b_l \frac{(s)_l}{l+s+\nu} \right)$$
Writing the terms of the second sum on a common denominator then gives
$$ \Gamma(s) \left( p_{n-1}(s) + \frac{q_{2m-2}(s)}{(s+\nu)_{m}} \right)$$
in terms of polynomials $p_{n-1}(s)$ and $q_{2m-2}(s)$ with their subscript as maximal degree. Finally, we take out the denominator $(s+\nu)_{m}$ to obtain
$$ \hat{F}_{n,m}(s) = \frac{\Gamma(s)}{(s+\nu)_{m}} \left( p_{n-1}(s) (s+\nu)_{m} + q_{2m-2}(s) \right).$$
Observe that the degree of the polynomial between the parentheses is less than or equal to $\max\{n+m-1,2m-2\}$. The restriction $m-1\leq n$ now implies that this maximum is always equal to $n+m-1$. On the other hand, we know that $\hat{F}_{n,m}(s)$ has zeros at the $n+m-1$ values $s=\alpha+1,\dots,\alpha+n+m-1$ because of the orthogonality conditions that the type I function $F_{n,m}$ satisfies. The polynomial between the brackets is therefore necessarily a non-zero scalar multiple of $(\alpha+1-s)_{n+m-1}$. Note that $\hat{F}_{n,m}(\alpha+n+m) \neq 0$. The normalization condition \eqref{I_norm} corresponds to taking the scalar $\mathcal{F}_{n,m}=1/\hat{F}_{n,m}(\alpha+n+m)$, which readily leads to the stated expression in \eqref{I_norm_ct}.
\end{proof}

We will compare the above to the Mellin transform of a Laguerre polynomial multiplied by a gamma density, which will be computed below. We will denote by $L_n^{(\alpha)}(x)$ the Laguerre polynomial of degree $n$ with leading coefficient $(-1)^n/n!$ (as in \cite[Chapter 5]{Szego}). It satisfies the orthogonality conditions
$$\int_0^\infty L_n^{(\alpha)}(x) e^{-x} x^{k+\alpha}  dx = 0,\quad k=0,\dots,n-1,$$
which will also be implied by its Mellin transform.

\begin{lem}  \label{Laguerre_transf}
The Mellin transform of $L_n^{(\alpha)}(x)e^{-x}$ is given by
    $$ \int_0^\infty L_n^{(\alpha)}(x) e^{-x} x^{s-1}\, dx = \Gamma(s) \frac{(\alpha+1-s)_n}{n!},\quad s > 0.$$
\end{lem}
\begin{proof}
We may use the Rodrigues formula for Laguerre polynomials (see \cite[Eq. (5.1.5)]{Szego})
\begin{equation} \label{Laguerre_Rodr}
    L_n^{(\alpha)}(x) = \frac{x^{-\alpha} e^{x}}{n!} \frac{d^n}{dx^n}\left[x^{n+\alpha} e^{-x}\right]
\end{equation}
and perform integration by parts $n$ times to get
$$ \int_0^\infty L_n^{(\alpha)}(x) e^{-x} x^{s-1} dx  = \frac{(-1)^n}{n!} \int_0^\infty x^{\alpha+n} e^{-x} \left( \frac{d^n}{dx^n} x^{s-\alpha-1}\right) dx .$$
After working out the derivatives, this becomes
$$\int_0^\infty L_n^{(\alpha)}(x) e^{-x} x^{s-1} dx = \frac{(\alpha+1-s)_n}{n!} \int_0^\infty x^{s-1} e^{-x} dx. $$
The remaining integral is the gamma function $\Gamma(s)$.
\end{proof}

The type I function can now be determined explicitly; it arises as the Mellin convolution of a Laguerre polynomial multiplied by a gamma density and a beta density. The Mellin convolution of two appropriate functions $f$ and $g$ is given by
    $$ (f\ast g)(x) = \int_0^\infty f(t) g(x/t) \frac{dt}{t}.  $$
Its key property is that its Mellin transform is multiplicative in terms of the Mellin transforms of the input functions
    $$ \int_0^\infty (f\ast g)(x) x^{s-1} dx = \left(\int_0^\infty f(x) x^{s-1} dx \right) \cdot \left(\int_0^\infty g(x) x^{s-1} dx \right).$$

\begin{thm}  \label{I_conv}
Suppose that $n+1\geq m$. The normalized type I function $F_{n,m}^{(I)}$ is given by $(n+m-1)!\mathcal{F}_{n,m}F_{n,m}/(m-1)!$, where $\mathcal{F}_{n,m}$ is the constant in \eqref{I_norm_ct} and
\begin{equation} \label{F}
   F_{n,m}(x) = \int_x^{\infty}  L_{n+m-1}^{(\alpha)}(t) e^{-t} \Bigl(1- \frac{x}{t} \Bigr)^{m-1} \Bigl(\frac{x}{t} \Bigr)^{\nu} \, \frac{dt}{t}.
\end{equation}
\end{thm}
\begin{proof}
The Mellin transform of the normalized type I function $F_{n,m}^{(I)}$, see Lemma \ref{I_Mellin_lem}, arises as a non-zero scalar multiple of the product of the Mellin transform of a Laguerre polynomial multiplied with a gamma density
$$ \int_0^\infty L_{n+m-1}^{(\alpha)}(x) e^{-x} x^{s-1} dx = \Gamma(s) \frac{(\alpha+1-s)_{n+m-1}}{(n+m-1)!},\quad s > 0 , $$
see Lemma \ref{Laguerre_transf}, and of a beta density
$$ \int_0^\infty (1-x)^{m-1} x^{\nu+s-1} \chi_{[0,1]}(x) dx = B(s+\nu,m) = \frac{(m-1)!}{(s+\nu)_m},\quad s > -\nu. $$
The multiplicative property of the Mellin convolution and the uniqueness of the Mellin transform then imply that $F_{n,m}^{(I)}$ is given by the Mellin convolution of the functions\break $(n+m-1)! L_{n+m-1}^{(\alpha)}(x)e^{-x}$ and $(1-x)^{m-1} x^{\nu} \chi_{[0,1]}(x)/(m-1)!$, which gives the integral in \eqref{F}.
\end{proof}

By working out the convolution formula for the type I function in \eqref{F}, we may obtain explicit expressions for the underlying type I polynomials $A_{n,m}$ and $B_{n,m}$. However, since this is rather tedious and not very insightful, we won't do it here. Later in Proposition \ref{IB_2F2}, we will be able to derive an expression for $B_{n,m}$ in a more elegant way.
\medbreak

We can use the convolution formula for the type I function to derive a Rodrigues formula for it. The relevant weight functions are those of the associated Nikishin system, i.e. $w_1$ and $w_2/w_1$.

\begin{thm}
Suppose that $n+1 \geq m$. The type I function in (\ref{F}) is given by the Rodrigues formula
\begin{equation} \label{I_Rodr}
    F_{n,m}(x) = \frac{x^{-\alpha}}{(n+m-1)!(\nu+1)_{m-1}} \frac{d^{n+m-1}}{dx^{n+m-1}}\left[x^{n+m-1}w_1(x)\frac{d^{m-1}}{dx^{m-1}}\left[x^{m-1} \frac{w_2(x)}{w_1(x)} \right] \right] .
\end{equation}
\end{thm}
\begin{proof}
The result will follow from the Rodrigues formula for the underlying Laguerre polynomial. We first perform the change of variables $t\mapsto xt$ in \eqref{F} to get
    $$ F_{n,m}(x) = \int_1^\infty L_{n+m-1}^{(\alpha)}(xt) e^{-xt} \left(1-\frac{1}{t}\right)^{m-1} \frac{dt}{t^{\nu+1}}. $$
From the Rodrigues formula $\eqref{Laguerre_Rodr}$ for $L_{n+m-1}^{(\alpha)}(x)$ and the chain rule, it follows that 
    $$ L_{n+m-1}^{(\alpha)}(xt) e^{-xt} = \frac{(xt)^{-\alpha}}{(n+m-1)!} \frac{d^{n+m-1}}{dx^{n+m-1}}\left[(xt)^{n+m-1+\alpha}e^{-xt}\right] t^{-n-m+1}, $$
or after rearranging some factors depending on $t$,
$$ L_{n+m-1}^{(\alpha)}(xt) e^{-xt} = \frac{x^{-\alpha}}{(n+m-1)!} \frac{d^{n+m-1}}{dx^{n+m-1}}\left[x^{n+m-1+\alpha}e^{-xt}\right] .$$
We can use this to take $n+m-1$ derivatives outside of the integral
    $$ F_{n,m}(x) = \frac{x^{-\alpha}}{(n+m-1)!} \frac{d^{n+m-1}}{dx^{n+m-1}}\left[ x^{n+m-1+\alpha} \int_{1}^{\infty} e^{-xt} \left(1-\frac{1}{t}\right)^{m-1} \frac{dt}{t^{\nu+1}} \right]$$
(the integrand is smooth for $x\in\R$ and $t\geq 1$, so an application of the Leibniz integral rule is justified). The remaining integral is related to the confluent hypergeometric function
    $$ U(a,b,x) = \frac{1}{\Gamma(a)} \int_{0}^{\infty} e^{-xt} t^{a-1} (1+t)^{b-a-1} \, dt, $$
which is defined for $x,a>0$ and $b\in\R$ (see \cite[Eq. 13.4.4]{DLMF}), via the change of variables $t\mapsto 1+t$
    $$ \int_{1}^{\infty} e^{-xt} \left(1-\frac{1}{t}\right)^{m-1} \frac{dt}{t^{\nu+1}} = (m-1)! e^{-x} U(m,-\nu+1;x). $$
The latter can be written in terms of
    $$ U(1,-\nu+1;x) = e^{x} \int_{1}^{\infty} e^{-xt} t^{-\nu-1} \, dt = e^{x} E_{\nu+1}(x) $$
by making use of the differentiation formula (see \cite[Eq. 13.3.24]{DLMF})
    $$ \frac{d^{n}}{dz^{n}}\left[ x^{n+a-1} U(a,b;z) \right] = (a)_n (a-b+1)_n x^{a-1} U(a+n,b;z). $$
Indeed, the special case 
    $$ \frac{d^{m-1}}{dx^{m-1}}\left[ x^{m-1} U(1,-\nu+1;x) \right] = (m-1)! (\nu+1)_{m-1} U(m,-\nu+1;x)$$
leads to the desired result.
\end{proof}

\subsection{Hypergeometric series}

The Mellin transform of the type I function can be used to derive a representation in terms of certain generalized hypergeometric series. The generalized hypergeometric function ${}_pF_q$ is a function that depends on two tuples $(a_1,\dots,a_p)\in\C^p$ and $(b_1,\dots,b_q)\in(\C\backslash\Z_{\leq 0})^q$ and is defined, for $p\leq q+1$, by the power series
    $$ {}_pF_q \left( \begin{array}{c} a_1,\dots,a_p \\ b_1,\dots,b_q \end{array}; x \right) = \sum_{k=0}^{\infty} \frac{(a_1)_k \dots (a_p)_k}{(b_1)_k \dots (b_q)_k} \frac{x^k}{k!},  $$
see, e.g., \cite[Chapter 16]{DLMF} for some basic properties.

\begin{thm} \label{I_2F2}
Suppose that $n+1\geq m$ and $\nu\not\in\Z$. The type I function with Mellin transform (\ref{I_Mellin}) arises as the following combination of ${}_2F_2$ hypergeometric series
$$\begin{aligned}
    F_{n,m}(x) = &\frac{(\alpha+1)_{n+m-1}}{(\nu)_m} {}_2F_2 \left( \begin{array}{c} n+m+\alpha,-m-\nu+1 \\ \alpha+1, -\nu+1 \end{array}; -x \right) \\
    &+ \frac{(\alpha+\nu+1)_{n+m-1} \Gamma(-\nu)}{(m-1)!} {}_2F_2 \left( \begin{array}{c} -m+1,n+m+\alpha+\nu \\ \nu+1,\alpha+\nu+1 \end{array}; -x \right) x^\nu .
\end{aligned}$$
\end{thm}
\begin{proof}
It will be more convenient to work with the Mellin transform of the complete type I function $x^{\alpha} F_{n,m}(x)$, which is
    $$ \int_0^\infty F_{n,m}(x) x^{\alpha+s-1} dx = \frac{\Gamma(s+\alpha)}{(s+\alpha+\nu)_m} (1-s)_{n+m-1},\quad s>-\min\{\alpha,\alpha+\nu\}. $$
Recall that the gamma function has an analytic continuation to $\C\backslash\Z_{\leq 0}$ that has simple poles at all $k\in\Z_{\leq 0}$ with residue $(-1)^k/k!$. We can therefore extend the above function in $s$ analytically to the complex plane except for the simple poles $-\alpha,-\alpha-1,\dots$ and $-\alpha-\nu,\dots,-\alpha-\nu-m+1$ (here we use that $\nu\not\in\Z$). The described poles all lie in an interval $(-\infty,c)$ for some $c\in\R$ that satisfies $-\min\{\alpha,\alpha+\nu,0\}<c<1$ (here we use that $\alpha,\alpha+\nu>-1$). The original function can then be recovered via the inverse Mellin transform
    $$ x^\alpha F_{n,m}(x) = \int_{c-i\infty}^{c+i\infty} \frac{\Gamma(s+\alpha)}{(s+\alpha+\nu)_m} (1-s)_{n+m-1} x^{-s} ds$$
by making use of the residue theorem. For this we can consider a counterclockwise oriented rectangle that connects the points $c\pm iN$ and $(c-M)\pm iN$, add the residues of the poles in this region and let $N,M\to\infty$. The residue of the integrand at $s=-\alpha-k$ with $k\in\Z_{\geq 0}$ is
    $$ \frac{(-1)^k}{k!} \frac{(k+\alpha+1)_{n+m-1}}{(-k+\nu)_m} x^{\alpha+k} = \frac{(-1)^k}{k!} \frac{(n+m+\alpha)_k (\alpha+1)_{n+m-1} (-m-\nu+1)_k }{(\alpha+1)_k (-\nu+1)_k (\nu)_m} x^{\alpha+k} $$
and the residue at $s=-\alpha-\nu-k$ with $k\in\{0,\dots,m-1\}$ is
$$ \begin{aligned} \frac{\Gamma(-k-\nu)}{(-k)_k (1)_{m-k-1}} & (k+\alpha+\nu+1)_{n+m-1} x^{\alpha+\nu+k} \\
    &= \frac{(-1)^k}{k!} \frac{\Gamma(-\nu) (n+m+\alpha+\nu)_k (\alpha+\nu+1)_{n+m-1} (-m+1)_k}{(\nu+1)_k (\alpha+\nu+1)_k (m-1)!} x^{\alpha+\nu+k}.
\end{aligned} $$
These are exactly the terms in the stated expression for $F_{n,m}$.
\end{proof}

The fact that part of the above expression is a polynomial times $x^{\nu}$ is to be expected, because of the following expansions for the weight functions (see \cite[Eq. 8.19.10]{DLMF}) of which $F_{n,m}$ is a polynomial combination:
    $$ e^{-x} = {}_0F_0 \left( \begin{array}{c} - \\ - \end{array}; -x \right), \quad E_{\nu+1}(x) = \Gamma(-\nu) x^\nu + \frac{1}{\nu} \, {}_1F_1 \left( \begin{array}{c} -\nu \\ -\nu+1 \end{array}; -x \right). $$
It also suggests that this polynomial will be related to $B_{n,m}(x)$. We will make this rigorous in what follows.
\medbreak

The connection to certain ${}_2F_2$ hypergeometric series allows us to obtain many properties of the type I function. For example, we can use it to determine a differential equation. 

\begin{prop} \label{I_diff_eq}
Suppose that $n+1\geq m$ and $\nu\not\in\Z$. Every type I function $y=F_{n,m}(x)$ satisfies the differential equation
$$ x^2 y''' + x (\alpha-\nu+3+x) y'' + ((\alpha+1)(-\nu+1)+x(n+\alpha-\nu+2))y' + (n+m+\alpha)(-m-\nu+1)y = 0. $$
\end{prop}
\begin{proof}
In general, it is known that a ${}_pF_q$ hypergeometric series satisfies a linear differential equation of order $\max\{p,q+1\}$, see \cite[Eq. 16.8.3]{DLMF} or \cite[\S 16.8 (ii)]{NIST}. For the ${}_2F_2$ hypergeometric series
    $$ {}_2F_2 \left( \begin{array}{c} a_1,a_2 \\ b_1, b_2 \end{array}; -x \right), $$
this differential equation becomes
\begin{equation} \label{2F2_diff_eq}
    x^2 y''' + x (b_1+b_2+1+x) y'' + (b_1b_2+x(a_1+a_2+1))y' + a_1a_2 y = 0.
\end{equation}
There are two other fundamental solutions and they are given by
    $$ x^{1-b_1} {}_2F_2 \left( \begin{array}{c} a_1-b_1+1,a_2-b_1+1 \\ -b_1+2, b_2-b_1+1 \end{array}; -x \right),\quad x^{1-b_2} {}_2F_2 \left( \begin{array}{c} a_1-b_2+1,a_2-b_2+1 \\ b_1-b_2+1, -b_2+2 \end{array}; -x \right),   $$
see \cite[Eq. 16.8.6]{DLMF}. The stated differential equation is the differential equation that corresponds to the hypergeometric series in the first term of the formula in Theorem \ref{I_2F2}. In this case, the two other solutions are of the form
    $$ x^{-\alpha} {}_2F_2 \left( \begin{array}{c} n+m,-m-\alpha-\nu+1 \\ -\alpha+1, -\alpha-\nu+1 \end{array}; -x \right), \quad 
    x^{\nu} {}_2F_2 \left( \begin{array}{c} n+m+\alpha+\nu,-m+1 \\ \alpha+\nu+1, \nu+1 \end{array}; -x \right).$$    
The type I function $F_{n,m}(x)$ thus arises as a linear combination of two solutions of a linear differential equation and is therefore a solution itself as well.
\end{proof}

We can use the differential equation for the type I function to obtain an explicit expression for the type I polynomial $B_{n,m}(x)$.

\begin{prop} \label{IB_2F2}
Suppose that $n+1\geq m$ and $\nu\not\in\Z$. Then
$$B_{n,m}(x) = B_{n,m}(0) \cdot {}_2F_2 \left( \begin{array}{c} -m+1 , n+m+\alpha+\nu \\ \nu+1 , \alpha+\nu+1 \end{array}; -x \right).$$
\end{prop}
\begin{proof}
Consider the differential equation in Proposition \ref{I_diff_eq}, but take the derivatives of $B_{n,m}(x)E_{\nu+1}(x)$ in $F_{n,m}(x)=A_{n,m}(x)e^{-x}+B_{n,m}(x) E_{\nu+1}(x)$ via $(B_{n,m}(x)x^{\nu})\cdot\Gamma_{-\nu}(x)$ and use that $(\Gamma_{-\nu}(x))'=-x^{-\nu-1}e^{-x}$. Since the functions $e^{-x}$ and $\Gamma_{-\nu}(x)$ are linearly independent over the polynomials (compare for example their Mellin transforms), $B_{n,m}(x)x^{\nu}$ itself must be a solution to the differential equation as well. Only one of the fundamental solutions listed in the proof of Proposition \ref{I_diff_eq} is of the form $x^{\nu}$ times a polynomial, hence we must have the stated result.
\end{proof}

\subsection{Recurrence relations}
It follows from the general theory of multiple orthogonal polynomials that type I (and type II) multiple orthogonal polynomials satisfy certain nearest neighbor recurrence relations, see, e.g., \cite{WVA}. For the normalized type I functions here, these are of the form
\begin{eqnarray}
  xF_{n,m}^{(I)}(x) &=& F_{n-1,m}^{(I)}(x)  + c_{n-1,m} F_{n,m}^{(I)}(x) + a_{n,m} F_{n+1,m}^{(I)}(x) + b_{n,m}F_{n,m+1}^{(I)}(x),    \label{NNRR_I_1} \\
  xF_{n,m}^{(I)}(x) &=& F_{n,m-1}^{(I)}(x) + d_{n,m-1} F_{n,m}^{(I)}(x) +  a_{n,m} F_{n+1,m}^{(I)}(x) + b_{n,m}F_{n,m+1}^{(I)}(x).    \label{NNRR_I_2}
\end{eqnarray}

We will determine the coefficients in two ways. Our first proof will be based on the three-term recurrence relation for Laguerre polynomials, which is (see \cite{Szego})
\begin{equation} \label{Laguerre_RR}
    -xL_n^{(\alpha)}(x) = (n+1) L_{n+1}^{(\alpha)}(x) - (2n+\alpha+1)L_n^{(\alpha)}(x) + (n+\alpha) L_{n-1}^{(\alpha)}(x).
\end{equation}
In the second proof, we will only use the Mellin transforms of the type I functions. It will be more straightforward but less constructive.
\medbreak

The first proof requires the following connection formula. A posteriori, it can also be obtained by subtracting the two nearest neighbor recurrence relations \eqref{NNRR_I_1}--\eqref{NNRR_I_2}.

\begin{prop}
One has
 $$ F_{n,m-1}^{(I)}(x)-F_{n-1,m}^{(I)}(x) = \frac{(n+m-1)(n+m+\alpha-1)(n+m+\alpha+\nu-1)}{(n+2m+\alpha+\nu-2)(n+2m+\alpha+\nu-1)} F_{n,m}^{(I)}(x). $$
\end{prop}
\begin{proof}
In terms of the type I functions with Mellin transforms \eqref{I_Mellin}, we have 
$$ \hat{F}_{n-1,m}(s) = \frac{1}{\alpha+n+m-s-1} \hat{F}_{n,m}(s)$$
and 
$$ \hat{F}_{n,m-1}(s) = \frac{s+\nu+m-1}{\alpha+n+m-s-1} \hat{F}_{n,m}(s). $$
Hence
    $$ (\alpha+\nu+n+2m-2) \hat{F}_{n-1,m}(s) - \hat{F}_{n,m-1}(s) = \hat{F}_{n,m}(s) ,$$
which for the normalized functions becomes
    $$ (n+2m+\alpha+\nu-2) \frac{\mathcal{F}_{n,m}}{\mathcal{F}_{n-1,m}} \hat{F}_{n-1,m}^{(I)}(x) - \frac{\mathcal{F}_{n,m}}{\mathcal{F}_{n,m-1}} \hat{F}_{n,m-1}^{(I)}(x) = \hat{F}_{n,m}^{(I)}(x). $$
If we plug in the values of the ratios
\begin{eqnarray*}
    \frac{\mathcal{F}_{n,m}}{\mathcal{F}_{n-1,m}} &=& - \frac{(n + 2m + \alpha + \nu - 1)}{(n + m - 1)(n + m + \alpha - 1)(n + m + \alpha + \nu - 1)} , \label{ratio_norm_ct_1} \\
    \frac{\mathcal{F}_{n,m}}{\mathcal{F}_{n,m-1}} &=& - \frac{(n + 2m + \alpha + \nu - 2)(n + 2m + \alpha + \nu - 1)}{(n + m - 1)(n + m + \alpha - 1)(n + m + \alpha + \nu - 1)}, \label{ratio_norm_ct_2}
\end{eqnarray*}
the desired result follows after applying the inverse Mellin transform.
\end{proof}
\medbreak

We can now derive the nearest neighbor recurrence relations for the normalized type I functions.

\begin{thm} \label{NNRR_coeff}
The coefficients in the nearest neighbor recurrence relations \eqref{NNRR_I_1}--\eqref{NNRR_I_2} are given by  
$$\begin{aligned}
       a_{n,m} &= \frac{(n+m)(n+m+\alpha)(n+m+\alpha+\nu)}{n+2m+\alpha+\nu} , \\
       b_{n,m} &= - \frac{(n+m)(n+m+\alpha)(n+m+\alpha+\nu)m(m+\nu)(m+\alpha+\nu)}{(n+2m+\alpha+\nu-1)(n+2m+\alpha+\nu)^2(n+2m+\alpha+\nu+1)}, \\
       c_{n-1,m} &= a_{n,m}-a_{n-1,m} , \\
       d_{n,m-1} &=  a_{n,m}-a_{n,m-1}.
\end{aligned}$$
\end{thm}
\begin{proof} 
Consider the type I functions in \eqref{F}. We can pass the multiplication by $x$ through the convolution as follows
$$xF_{n,m}(x) =  \int_x^\infty \left[t L_{n+m-1}^{(\alpha)}(t)\right] e^{-t} \left[\frac{x}{t} \left( 1-\frac{x}{t} \right)^{m-1}\right] \left( \frac{x}{t} \right)^{\nu} \frac{dt}{t}.$$
If we then combine the three-term recurrence relation \eqref{Laguerre_RR} for Laguerre polynomials and the recurrence relation $x(1-x)^{m-1} = (1-x)^{m-1}-(1-x)^m$, we find
$$ \begin{aligned}
xF_{n,m}(x) =\ & (n+m) [F_{n,m+1}(x) - F_{n+1,m}] \\  
               & -\, (2n+2m+\alpha-1) [F_{n-1,m+1}(x) - F_{n,m}(x)] \\
               & +\, (n+m+\alpha-1)[F_{n-2,m+1} - F_{n-1,m}(x)],
\end{aligned} $$
For the normalized type I functions, this becomes
$$ \begin{aligned}
xF_{n,m}^{(I)}(x) =\ & \left[ \frac{m \mathcal{F}_{n,m}}{\mathcal{F}_{n,m+1}} F_{n,m+1}^{(I)}(x) - \frac{\mathcal{F}_{n,m}}{\mathcal{F}_{n+1,m}} F_{n+1,m}^{(I)}(x) \right] \\  
               & -\ (2n+2m+\alpha-1) \left[ \frac{m \mathcal{F}_{n,m}}{\mathcal{F}_{n-1,m+1}}F_{n-1,m+1}^{(I)}(x)-F_{n,m}^{(I)}(x) \right] \\
               & +\ (n+m-1)(n+m+\alpha-1)\left[ \frac{m \mathcal{F}_{n,m}}{\mathcal{F}_{n-2,m+1}} F_{n-2,m+1}^{(I)}(x)-\frac{\mathcal{F}_{n,m}}{\mathcal{F}_{n-1,m}} F_{n-1,m}^{(I)}(x) \right].
\end{aligned} $$
In order to obtain the nearest neighbor recurrence relations, we need to get rid of the terms $F_{n-1,m+1}^{(I)}(x)$ and $F_{n-2,m+1}^{(I)}(x)$. We can do this by making use of the connection formula. One has
$$F_{n-2,m+1}^{(I)}(x) = F_{n-1,m}^{(I)}(x) - C_{n-1,m+1} F_{n-1,m+1}^{(I)}(x),$$
$$F_{n-1,m+1}^{(I)}(x) = F_{n,m}^{(I)}(x) -  C_{n,m+1} F_{n,m+1}^{(I)}(x).$$
The coefficient of the function $F_{n-1,m}^{(I)}(x)$ then becomes
    $$  (n+m-1)(n+m+\alpha+\nu-1) \left[\frac{m \mathcal{F}_{n,m}}{\mathcal{F}_{n-2,m+1}} - \frac{\mathcal{F}_{n,m}}{\mathcal{F}_{n-1,m}} \right]  = 1. $$
The other coefficients become, using the same notation as in \eqref{NNRR_I_1},
$$\begin{aligned}
       a_{n,m} = &-\frac{\mathcal{F}_{n,m}}{\mathcal{F}_{n+1,m}} , \\
       b_{n,m} = &\,\frac{m \mathcal{F}_{n,m}}{\mathcal{F}_{n,m+1}} + (2n+2m+\alpha-1) \frac{m\mathcal{F}_{n,m}}{\mathcal{F}_{n-1,m+1}} C_{n,m+1} \\
       &+ (n+m-1)(n+m+\alpha-1) \frac{m\mathcal{F}_{n,m}}{\mathcal{F}_{n-2,m+1}} C_{n-1,m+1} C_{n,m+1}, \\
       c_{n-1,m} = &-\ (2n+2m+\alpha-1) \left[\frac{m \mathcal{F}_{n,m}}{\mathcal{F}_{n-1,m+1}}-1\right] \\
       &+ (n+m-1)(n+m+\alpha-1) \frac{m \mathcal{F}_{n,m}}{\mathcal{F}_{n-2,m+1}} C_{n-1,m+1} C_{n-1,m+1} .
\end{aligned}$$
After some computations, we obtain the expressions described in the theorem. For \eqref{NNRR_I_2}, we can use the connection formula to replace $F_{n-1,m}^{(I)}(x)$ in \eqref{NNRR_I_1} by $F_{n,m-1}^{(I)}(x)$ and an extra term containing $F_{n,m}^{(I)}(x)$.
\end{proof}
\medbreak

\textit{Alternative proof of Theorem \ref{NNRR_coeff}.}
Consider the type I functions with Mellin transform \eqref{I_Mellin}. We may take the Mellin transform of the linear form
    $$ c_1 F_{n-1,m}(x)  + c_2 F_{n,m}(x) + c_3 F_{n+1,m}(x) + c_4 F_{n,m+1}(x) $$
and show that for suitable values of the coefficients it will be equal to the Mellin transform of $xF_{n,m}(x)$. By converting to the normalized type I functions, we should then end up with the described recurrence coefficients. After working out the described Mellin transforms, we would like to have 
$$\begin{aligned}
    &c_1(s + \nu + m) + c_2(\alpha - s + n + m - 1)(s + \nu + m) \\
    &+ c_3(\alpha - s + n + m - 1)(\alpha - s + n + m)(s + \nu + m) \\
    &+ c_4(\alpha - s + n + m - 1)(\alpha - s + n + m) \\
    &= s(s + \nu)(\alpha - s).
\end{aligned}$$
By comparing the coefficients of $s$ on both sides of the equation, we obtain four equations for the four unknown $c_j$. We can solve this system of equations via a computer algebra system. We find
$$\begin{aligned}
    c_1 &= -\frac{(n + m - 1)(n + m + \alpha + \nu - 1)(n + m + \alpha - 1)}{n + 2m + \alpha + \nu - 1}, \\
    c_2 &= 2n + m + \alpha - 1 + \frac{(m + n)m(m + \nu)}{n + 2m + \alpha + \nu} - \frac{(m + n - 1)m(m + \nu)}{n + 2m + \alpha + \nu - 1}, \\
    c_3 &= -1, \\
    c_4 &= \frac{m(m + \alpha + \nu)(m + \nu)}{(\alpha + \nu + n + 2m)(n + 2m + \alpha + \nu - 1)}.
\end{aligned}$$
The associated recurrence relation for the normalized functions then has coefficients
$$\begin{aligned}
    c_1^{(I)} &= c_1 \mathcal{F}_{n,m}/\mathcal{F}_{n-1,m} = - c_1 \frac{(n + 2m + \alpha + \nu - 1)}{(n + m - 1)(n + m + \alpha - 1)(n + m + \alpha + \nu - 1)} = 1, \\
    c_2^{(I)} &= c_2 = a_{n,m} - a_{n-1,m} , \\
    c_3^{(I)} &= c_3 \mathcal{F}_{n,m}/\mathcal{F}_{n+1,m} = - c_3 \frac{(n + m)(n + m + \alpha)(n + m + \alpha + \nu)}{n + 2m + \alpha} = a_{n,m}, \\
    c_4^{(I)} &= c_4 \mathcal{F}_{n,m}/\mathcal{F}_{n,m+1} = - c_4 \frac{(n + m)(n + m + \alpha)(n + m + \alpha + \nu)}{(n + 2m + \alpha + \nu)(n + 2m + \alpha + \nu + 1)} = b_{n,m}.
\end{aligned}$$

In order to obtain the formula for $c_2$, we used the fact that 
    $$ a_{n,m} = (n+m)(n+\alpha) + \frac{(n+m)m(m+\nu)}{n+2m+\alpha+\nu},$$
which follows from a partial decomposition of $a_{n,m}$ as a function in $\alpha$.
\qed

\section{Type II multiple orthogonal polynomials} \label{STII}

A type II multiple orthogonal polynomial for the two weight functions $(x^\alpha e^{-x},x^\alpha E_{\nu+1}(x))$ that corresponds to the multi-index $(n,m)$ is a polynomial $P_{n,m}$ with $\deg P_{n,m} = n+m$ that satisfies the orthogonality relations
\begin{eqnarray*}
   \int_0^\infty P_{n,m}(x) e^{-x} x^{k+\alpha} dx = 0,\quad k=0,\dots,n-1, \\
   \int_0^\infty P_{n,m}(x) E_{\nu+1}(x) x^{k+\alpha} dx = 0,\quad k=0,\dots,m-1.
\end{eqnarray*}
Such polynomials (having maximal degree) exist for every multi-index because of the perfectness of the system and they are unique up to a non-zero scalar multiplication. The normalized type II polynomial, which is monic, will be denoted by $P_{n,m}^{(II)}$. Similarly as before, we will only consider multi-indices $(n,m)$ with $n+1\geq m$.

\subsection{Differentiation formula}
The type II polynomials will also be connected to the Laguerre polynomials; now the connection will come in the form of a differential formula. In order to obtain this formula, we require the following reduction property of the type II polynomials. Note that we explicitly wrote down the dependence of the type II polynomial $P_{n,m}=P_{n,m}^{\alpha,\nu}$ on the parameters.

\begin{lem} \label{II_red}
Suppose that $n\geq m+1$. Then
    $$ \frac{d}{dx}\left[x^{\alpha+\nu} P_{n,m}^{\alpha,\nu}(x)\right] \propto x^{\alpha+\nu-1} P_{n-1,m+1}^{\alpha,\nu-1}(x), $$
where the symbol $\propto$ means equality up to a non-zero scalar multiplication.
\end{lem}
\begin{proof}
After working out the derivative, it is clear that the left hand side is of the form $x^{\alpha+\nu-1}$ times a polynomial of degree $n+m$. It then remains to verify that it satisfies the required orthogonality conditions, i.e.,
    $$ \int_0^\infty \frac{d}{dx}\left[x^{\alpha+\nu} P_{n,m}^{\alpha,\nu}(x)\right] e^{-x} x^{k-\nu+1} dx = 0,\quad k=0,\dots,n-2,  $$
    $$ \int_0^\infty \frac{d}{dx}\left[x^{\alpha+\nu} P_{n,m}^{\alpha,\nu}(x)\right] E_{\nu}(x) x^{k-\nu+1} dx = 0,\quad k=0,\dots,m. $$
We can rewrite both integrals via integration by parts. The first one becomes
$$\left[P_{n,m}^{\alpha,\nu}(x) e^{-x} x^{k+\alpha+1}\right]_0^{\infty} -\int_0^\infty x^{\alpha+\nu} P_{n,m}^{\alpha,\nu}(x) \frac{d}{dx}\left[e^{-x} x^{k-\nu+1}\right] dx.$$
The restriction $\alpha>-1$ implies that the first part will vanish. After working out the derivative, we thus get
$$ -\int_0^\infty P_{n,m}^{\alpha,\nu}(x) e^{-x} x^{k+\alpha} (-x+(k-\nu+1)) dx.$$
This vanishes for $k=0,\dots,n-2$ due to the orthogonality of $P_{n,m}^{\alpha,\nu}$ with respect to the first weight.
The other integral becomes
$$\left[P_{n,m}^{\alpha,\nu}(x) E_{\nu}(x) x^{k+\alpha+1}\right]_0^{\infty} -\int_0^\infty x^{\alpha+\nu} P_{n,m}^{\alpha,\nu}(x) \frac{d}{dx}\left[E_{\nu}(x) x^{k-\nu+1}\right] dx,$$
where the first part vanishes again because $\alpha>-1$. It is convenient to use the relation $E_{\nu}(x) = x^{\nu-1} \Gamma_{-\nu+1}(x)$ when computing the derivative
    $$ \frac{d}{dx}\left[E_{\nu}(x) x^{k-\nu+1}\right] = -x^{k-\nu} e^{-x} + k x^{k-1} \Gamma_{-\nu+1}(x) 
    = - x^{k-\nu} (e^{-x} - k E_{\nu}(x)). $$
We can use the recurrence relation $\nu E_{\nu+1}(x)+xE_{\nu}(x)=e^{-x}$, see \cite[Eq. 8.19.12]{DLMF}, to write this in terms of $E_{\nu+1}(x)$. The integral then becomes
    $$ \int_0^\infty P_{n,m}^{\alpha,\nu}(x) x^{k+\alpha-1} ((x-k)e^{-x}+k\nu E_{\nu+1}(x)) dx.$$
For $k=0$, this is equal to
$$\int_0^\infty P_{n,m}^{\alpha,\nu}(x) x^{\alpha} e^{-x} dx =0,$$
while for $k\geq 1$, we have
    $$ \int_0^\infty P_{n,m}^{\alpha,\nu}(x) x^{k+\alpha-1} (x-k)e^{-x} dx = 0, \qquad \text{ if } k\leq n-1$$
and 
    $$ \int_0^\infty P_{n,m}^{\alpha,\nu}(x) x^{k+\alpha-1} E_{\nu+1}(x) dx = 0, \qquad \text{ if } k\leq m.$$
Since $m\leq n-1$, the desired orthogonality conditions with respect to the second weight are satisfied as well.
\end{proof}

The previous property allows us to express the type II polynomial in terms of a Laguerre polynomial.   

\begin{thm}
Suppose that $n+1\geq m$. The monic type II multiple orthogonal polynomial $P_{n,m}^{(II)}$ is given by $\mathcal{P}_{n,m}\cdot P_{n,m}$, where
\begin{equation} \label{II_diff}
    P_{n,m}(x) = x^{-\alpha-\nu} \frac{d^{m}}{dx^m}\left[x^{m+\alpha+\nu} L_{n+m}^{(\alpha)}(x)\right]. 
\end{equation}
and
$$\mathcal{P}_{n,m} = (-1)^{n+m} (n+m)! \frac{(\alpha+\nu+1)_{n+m}}{(\alpha+\nu+1)_{n+2m}} .$$
\end{thm}
\begin{proof}
We have already shown that
    $$ x^{\alpha+\nu} P_{n,m}^{\alpha,\nu}(x) \propto \frac{d}{dx}\left[x^{\alpha+\nu+1}P_{n+1,m-1}^{\alpha,\nu+1}(x)\right],$$
for $n+1\geq m$, so if we extract $m-1$ additional derivatives, we get
    $$ x^{\alpha} P_{n,m}^{\alpha,\nu}(x) \propto \frac{d^m}{dx^m}\left[x^{\alpha+m}P_{n+m,0}^{\alpha,\nu+m}(x)\right].$$
The remaining polynomial is only orthogonal with respect to the first weight $x^{\alpha}e^{-x}$, hence we must have $P_{n+m,0}^{\alpha,\nu+m}(x)\propto L_{n+m}^{(\alpha)}(x)$. The leading coefficient of this Laguerre polynomial is $(-1)^{n+m}/(n+m)!$, so the leading coefficient of the monic type II polynomial is $\mathcal{P}_{n,m} (-1)^{n+m}(n+m)! (n+m+\alpha+\nu+1)_m$. This readily leads to the stated expression for the constant $\mathcal{P}_{n,m}$.
\end{proof}

Observe that if we would replace the Laguerre polynomial in this formula by its Rodrigues formula \eqref{Laguerre_Rodr}, we would obtain a Rodrigues formula for the type II polynomial as well.

\subsection{Hypergeometric series}

In what follows, we will show that the type II polynomials are hypergeometric polynomials. This property will be inherited from the underlying Laguerre polynomial, which is a hypergeometric polynomial as well. More precisely, we have (see \cite[Eq. (5.3.3)]{Szego})
$$ L_{n}^{(\alpha)}(x) = \frac{(\alpha+1)_n}{n!} \, {}_1F_1 \left( \begin{array}{c} -n \\ \alpha+1 \end{array}; x \right). $$

\begin{thm}
Suppose that $n+1\geq m$. The type II polynomial in \eqref{II_diff} arises as the generalized hypergeometric series
\begin{equation} \label{II_2F2}
    P_{n,m}(x) = \frac{(\alpha+1)_{n+m} (\alpha+\nu+1)_m}{(n+m)!} \, {}_2F_2 \left( \begin{array}{c} -n-m,m+\alpha+\nu+1 \\ \alpha+1, \alpha+\nu+1 \end{array}; x \right).
\end{equation}
\end{thm}
\begin{proof}
Use the previous lemma and the following differentiation formula for the generalized hypergeometric series
$$ \frac{d^n}{dz^n}\left[z^\gamma {}_pF_q \left( \begin{array}{c} a_1,\dots,a_p \\ b_1,\dots,b_q \end{array}; z \right) \right] = (\gamma-n+1)_n z^{\gamma-n} {}_{p+1}F_{q+1} \left( \begin{array}{c} \gamma+1,a_1,\dots,a_p \\ \gamma+1-n,b_1,\dots,b_q \end{array}; z \right), $$
see \cite[Eq. 16.3.2]{DLMF}.
\end{proof}

Note that this hypergeometric series is similar to the hypergeometric series that is used in the remaining fundamental solution of the differential equation in (the proof of) Proposition $\ref{I_diff_eq}$.
\medbreak

The previous result allows us to determine a differential equation for $P_{n,m}$.

\begin{prop}
Every type II polynomial $y=P_{n,m}(x)$ satisfies the third order differential equation
\begin{equation} \label{II_diff_eq}
    x^2 y''' + x(2\alpha+\nu+3-x) y'' + [(\alpha+1)(\alpha+\nu+1)+x(n-\alpha-\nu-2)]y' + (n+m)(m+\alpha+\nu+1)y = 0.
\end{equation}
\end{prop}
\begin{proof}
This is the differential equation associated to the ${}_2F_2$ hypergeometric series in \eqref{II_2F2}. In general, the hypergeometric series
    $$ {}_2F_2 \left( \begin{array}{c} a_1,a_2 \\ b_1, b_2 \end{array}; x \right) $$
satisfies the differential equation
    $$ x^2 y''' + x (b_1+b_2+1-x) y'' + (b_1b_2-x(a_1+a_2+1))y' - a_1a_2 y = 0, $$
see \cite[Eq. 16.8.3]{DLMF}, \cite[\S 16.8 (ii)]{NIST} or use \eqref{2F2_diff_eq} in a negative argument.
\end{proof}

\subsection{Recurrence relations}
The nearest neighbor recurrence relations for the monic type II polynomials are similar to those for the normalized type I functions \eqref{NNRR_I_1}--\eqref{NNRR_I_2}; one uses the same coefficients up to a shift 
\begin{eqnarray}
  xP_{n,m}^{(II)}(x) &=& P_{n+1,m}^{(II)}(x)  + c_{n,m} P_{n,m}^{(II)}(x) + a_{n,m} P_{n-1,m}^{(II)}(x) + b_{n,m}P_{n,m-1}^{(II)}(x),    \label{NNRR_II_1} \\
  xP_{n,m}^{(II)}(x) &=& P_{n,m+1}^{(II)}(x) + d_{n,m} P_{n,m}^{(II)}(x) +  a_{n,m} P_{n-1,m}^{(II)}(x) + b_{n,m}P_{n,m-1}^{(II)}(x).    \label{NNRR_II_2}
\end{eqnarray}
Below we will describe an explicit method to derive these recurrence relations from the three-term recurrence relation for the underlying Laguerre polynomials. We again require a connection formula.
\begin{lem} \label{NNRR_II_CF}
One has 
 $$ P_{n,m+1}^{(II)}(x)-P_{n+1,m}^{(II)}(x) = \frac{(n+m+\alpha+1)(n+m+\alpha+\nu+1)}{(n+2m+\alpha+\nu+1)(n+2m+\alpha+\nu+2)} P_{n,m}^{(II)}(x). $$
\end{lem}
\begin{proof}
Consider the hypergeometric series representations of the monic type II polynomials in \eqref{II_2F2} and compare the coefficients of each monomial $x^k$ on both sides of the equation. If we denote the coefficient of $x^k$ in a polynomial $P(x)$ by $P[k]$, we have
$$P_{n,m+1}^{(II)}[k] = \frac{(\alpha+n+m+1)(\alpha+\nu+n+m+1)(\alpha+\nu+m+k+1)}{(\alpha+\nu+n+2m+1)(\alpha+\nu+n+2m+2)(-n-m+k-1)} P_{n,m}^{(II)}[k],$$
$$P_{n+1,m}^{(II)}[k] = \frac{(\alpha+n+m+1)(\alpha+\nu+n+m+1)}{(\alpha+\nu+n+2m+1)(-n-m+k-1)} P_{n,m}^{(II)}[k].$$
Now use
$$ \frac{\alpha+\nu+m+k+1}{\alpha+\nu+n+2m+2} - 1 = \frac{-n-m+k-1}{\alpha+\nu+n+2m+2}$$
to get the stated result.
\end{proof}

\textit{Alternative proof of \eqref{NNRR_II_1}--\eqref{NNRR_II_2}.}
We may compute one of the derivatives in the differentiation formula \eqref{II_diff} via
    $$ P_{n,m}(x) = x^{-\alpha-\nu} \frac{d^m}{dx^m} \left[ x^{m+\alpha+\nu-1} \cdot xL_{n+m}^{(\alpha)}(x) \right], $$
so that after an application of the Leibniz rule, we get
    $$ P_{n,m}(x) = m! \sum_{k=0}^{m} \binom{m+\alpha+\nu-1}{m-k} \frac{x^{k}}{k!} \frac{d^k}{dx^k}\left[ xL_{n+m}^{(\alpha)}(x) \right].  $$
The three-term recurrence relation \eqref{Laguerre_RR} for Laguerre polynomials and the identity
    $$ \binom{m+\alpha+\nu-1}{m-k} = \binom{m+\alpha+\nu}{m-k} - \binom{m-1+\alpha+\nu}{m-1-k}  $$
then lead to the recurrence relation
$$ \begin{aligned}
    P_{n,m}(x) =\ & (n+m+1) [m P_{n+2,m-1}(x) - P_{n+1,m}(x)] \\
    &- (2n+2m+\alpha+1) [ m P_{n+1,m-1}(x) - P_{n,m}(x)] \\
    &+ (n+m+\alpha) [ m P_{n,m-1}(x) P_{n-1,m}(x)].
\end{aligned} $$
For the monic type II polynomials, this becomes
$$ \begin{aligned}
    P_{n,m}^{(II)}(x) =\ & \left[\frac{m\mathcal{P}_{n,m}}{\mathcal{P}_{n+2,m-1}} P_{n+2,m-1}^{(II)}(x) - \frac{\mathcal{P}_{n,m}}{\mathcal{P}_{n+1,m}} P_{n+1,m}^{(II)}(x)\right] \\
    &- (2n+2m+\alpha+1) \left[\frac{m\mathcal{P}_{n,m}}{\mathcal{P}_{n+1,m-1}} P_{n+1,m-1}^{(II)}(x) - P_{n,m}^{(II)}(x)\right] \\
    &+ (n+m)(n+m+\alpha) \left[\frac{m\mathcal{P}_{n,m}}{\mathcal{P}_{n,m-1}} P_{n,m-1}^{(II)}(x) - \frac{\mathcal{P}_{n,m}}{\mathcal{P}_{n-1,m}} P_{n-1,m}^{(II)}(x)\right].
\end{aligned} $$
We may use the connection formula (Lemma \ref{NNRR_II_CF}) to convert this to the nearest neighbor recurrence relations.
\qed

\subsection{Asymptotic behavior of the zeros}
In what follows, we will investigate the asymptotic behavior of the zeros of the type II polynomial $P_{\vec{n}}=P_{n_1,n_2}$, with multi-index $\vec{n}=(n_1,n_2)$, as $n=n_1+n_2\to\infty$. 
We will now use a multi-index $(n_1,n_2)$ instead of $(n,m)$ to avoid confusion with the degree
$n = n_1+n_2$ which tends to infinity.
The limit will be taken along multi-indices $\vec{n}$ with $(n_1/n,n_2/n)\to(q_1,q_2)$ as $n\to\infty$. We necessarily have $0\leq q_1,q_2\leq 1$ and $q_1+q_2=1$. The restriction that $n_1+1\geq n_2$ leads to the condition $q_1\geq q_2$.

Afterwards, we will compare the results to what is known for the monic type II multiple Laguerre polynomials of the first kind $L_{\vec{n}}^{(II\mid\alpha,\beta)}(x)$. These polynomials satisfy the following orthogonality conditions
\begin{eqnarray*}
   \int_0^\infty L_{\vec{n}}^{(II\mid\alpha,\beta)}(x) e^{-x} x^{k+\alpha} dx = 0, \quad k=0,\dots,n_1-1, \\
   \int_0^\infty L_{\vec{n}}^{(II\mid\alpha,\beta)}(x) e^{-x} x^{k+\beta} dx = 0, \quad k=0,\dots,n_2-1.
\end{eqnarray*}

\subsubsection{Smallest zeros}

We will focus first on the behavior of the smallest zeros of $P_{\vec{n}}$, which we may derive from its Mehler-Heine asymptotics. The Mehler-Heine asymptotics of several families of type II multiple orthogonal polynomials, including the multiple Laguerre polynomials of the first kind, were determined in \cite{WVA_MHA}. We will employ the same strategy used there.

\begin{prop}
For the type II polynomial \eqref{II_2F2}, one has
    $$ \lim_{n\to\infty} \frac{n!}{(\alpha+1)_{n}(\alpha+\nu+1)_{n_2}} P_{\vec{n}}(x/n^2) = {}_0F_2 \left( \begin{array}{c} - \\ \alpha+1, \alpha+\nu+1 \end{array}; -q_2 x \right).$$
As a consequence, if we denote the $k$-th zero of $P_{\vec{n}}(x)$ and of this hypergeometric series by $p_{k,\vec{n}}$ and $x_{k}$ respectively, we have $\lim_{n\to\infty} n^2 p_{k,\vec{n}} = x_{k}$.

\end{prop}
\begin{proof}
The first result follows immediately from the explicit formula \eqref{II_2F2} and
$$ \lim_{n\to\infty} \frac{(-n)_k}{n^k} = (-1)^k, \quad \lim_{n\to\infty} \frac{(n_2+\alpha+1)_k}{n^k} = q_2^k.$$
In order to obtain the second one, we can then apply Hurwitz' theorem.
\end{proof}
    
The associated result for the type II multiple Laguerre polynomials of the first kind is
$$ \lim_{n\to\infty} c_{\vec{n}}^{(\alpha,\beta)} L_{\vec{n}}^{(II\mid\alpha,\beta)}(x/n^2) = {}_0F_2 \left( \begin{array}{c} - \\ \alpha+1, \beta+1 \end{array}; -q_1q_2 x \right),$$
see \cite[Theorem 3]{WVA_MHA}. Thus, if we denote the $k$-th zero of $L_{\vec{n}}^{(II\mid\alpha,\beta)}$ by $l_{\vec{n},k}^{(\alpha,\beta)}$ and apply Hurwitz' theorem again, we can conclude that
    $$ \lim_{n\to\infty}\frac{p_{\vec{n},k}}{l_{\vec{n},k}^{(\alpha,\alpha+\nu)}} = q_1.$$
This shows that the smallest zeros of $P_{\vec{n}}$ will be closer to $0$ than the smallest zeros of $L_{\vec{n}}^{(II\mid\alpha,\alpha+\nu)}$ and that this effect strengthens for smaller values of $q_1$, i.e., if we consider more orthogonality conditions with respect to the exponential integral weight.

\subsubsection{Asymptotic zero distribution and largest zero}

We will now determine the asymptotic zero distribution of the polynomials $P_{\vec{n}}$ as $n\to\infty$. 
In \cite{Couss2VA}, one was able to determine the asymptotic zero distribution for several families of type II multiple orthogonal polynomials, including the multiple Laguerre polynomials of the first kind, for multi-indices on the step-line by investigating an underlying four-term recurrence relation. Essential to their approach was that the coefficients in this recurrence relations satisfied a particular asymptotic property. After using Lemma \ref{NNRR_II_CF} to write the nearest neighbor recurrence relations \eqref{NNRR_II_1}--\eqref{NNRR_II_2} for $P_{\vec{n}}$ as a four-term recurrence relation on the step-line, it can be verified that the coefficients do not satisfy the desired property. The asymptotic zero distribution of the polynomials $P_{\vec{n}}$ therefore doesn't fit in the class of distributions found in this article and we will require a more hands on approach to determine it (which was necessary anyway if we don't want to restrict to multi-indices on the step-line).
\medbreak

Since the pair $(w_1,w_2)$ forms a Nikishin system, and thus an AT-system, of weights supported on $[0,\infty)$ and $(-\infty,0]$, $P_{\vec{n}}$ will have $n$ real simple zeros in $[0,\infty)$ (see, e.g., \cite[Chapter 4, \S 4]{NikiSor}). This interval is unbounded, so we can expect that the (largest) zeros of $P_{\vec{n}}$ will flow away to $\infty$. In order to proceed, we need to find an appropriate scaling $\tilde{P}_{\vec{n}}(x)=P_{\vec{n}}(n^{\gamma}x)$ such that the zeros are in a compact interval $[0,c]$ with $c=c_{q_1,q_2}>0$. A suitable value of $\gamma$ is indicated by the coefficients of the nearest neighbor recurrence relations \eqref{NNRR_II_1}--\eqref{NNRR_II_2}; it is the (smallest) value of $\gamma$ for which the recurrence relations for the monic scaled polynomial has bounded coefficients. As $n\to\infty$, we have 
$$ a_{\vec{n}} = \BO(n^2), \quad b_{\vec{n}} = \BO(n^2), \quad c_{\vec{n}} = \BO(n), \quad d_{\vec{n}} = \BO(n), $$
which suggests taking $\gamma=1$.
\medbreak

The right scaling can be derived more rigorously via the following lemma. It also shows that we can assume that $c\leq 4$. Additionally, the proof gives an explicit method to prove the fact that $P_{\vec{n}}$ has $n$ simple real zeros in $[0,\infty)$.

\begin{lem}
For the largest zero $p_{n,\vec{n}}$ of $P_{\vec{n}}$, we have $\limsup_{n\to\infty} x_{n,\vec{n}}/n \leq 4$.
\end{lem}
\begin{proof}
We will show this by comparing the largest zero of $P_{\vec{n}}$ with the largest zero of the underlying Laguerre polynomial $L_n^{(\alpha)}$ in the differentiation formula \eqref{II_diff}. The idea will be that the operator 
$$P(x)\mapsto x^{-\alpha-\nu}\frac{d^m}{dx^m}\left[x^{m+\alpha+\nu} P(x) \right]$$
will pull the zeros of $L_n^{(\alpha)}(x)$ towards $0$. It is known from the general theory of orthogonal polynomials that $L_n^{(\alpha)}(x)$ has $n$ real simple zeros in $[0,\infty)$. If we apply the operator 
$$P(x)\mapsto x^{-\alpha-\nu}\frac{d}{dx}\left[x^{\alpha+\nu} P(x) \right]$$
once to $x^m L_n^{(\alpha)}(x)$, we get a polynomial with a zero of order $m-1$ at $0$ and of which the other $n$ zeros are real, simple, in $[0,\infty)$ and interlace with the zeros of the input polynomial $L_n^{(\alpha)}(x)$ and $0$ (due to Rolle's theorem). In particular, all of these zeros are smaller than the largest zero of $L_n^{(\alpha)}(x)$. Applying the operator $m-1$ additional times and repeating this argument, we recover $P_{\vec{n}}$ together with the conclusion that it has $n$ simple real zeros in $[0,\infty)$ all smaller than the largest zero of $L_n^{(\alpha)}(x)$. It is known that the largest zero of the Laguerre polynomial $L_n^{(\alpha)}(x)$ behaves as $4n$ as $n\to\infty$ (see, e.g., \cite[Thm.~6.31.2]{Szego}). 
Hence we get the estimate in the lemma.
\end{proof}

In what follows, we will use some of the techniques from \cite{LeursVA} to derive the asymptotic zero distribution of $\tilde{P}_{\vec{n}}(x)=P_{\vec{n}}(nx)$. Denote the zeros of $\tilde{P}_{\vec{n}}(x)$ by $x_{1,\vec{n}}< \dots< x_{n,\vec{n}}$. Their distribution is given by the probability measure
    $$ \mu_{\vec{n}} = \frac{1}{n} \sum_{j=1}^n \delta_{x_{j,\vec{n}}} $$
with Stieltjes transform
    $$ S_{\vec{n}}(z) = \int_0^c \frac{d\mu_{\vec{n}}(x)}{z-x} = \frac{1}{n} \frac{\tilde{P}'_{\vec{n}}(z)}{\tilde{P}_{\vec{n}}(z)}.$$
According to Helly's selection principle, the sequence of probability measures $(\mu_{\vec{n}})_{n\in\N}$ on $[0,c]$ has a subsequence $(\mu_{\vec{n}_k})_{k\in\N}$ that converges weakly to a probability measure $\mu=\mu_{q_1,q_2}$ on $[0,c]$. If this limit is independent of the subsequence, it is called the asymptotic zero distribution of $\tilde{P}_{\vec{n}}(x)$. 

Let $(\mu_{\vec{n}_k})_{k\in\N}$ be any subsequence of $(\mu_{\vec{n}})_{n\in\N}$ that converges weakly to $\mu$. Then it follows from the Grommer-Hamburger theorem (see, e.g., \cite{GerHill}) that $(S_{\vec{n}_k}(z))_{k\in\N}$ converges uniformly to the Stieltjes transform 
    $$ S(z)=\int_0^c \frac{d\mu(x)}{z-x}$$
on compact subsets of $\C\backslash[0,c]$. Moreover, $zS(z)\to 1$ as $z\to\infty$. 

We can proceed via the (third order) differential equation for $P_{\vec{n}}(x)$; we will use it to show that $S(z)$ satisfies a third order algebraic equation. Since this equation will be independent of the selected subsequence, we will be able to conclude that $S(z)$ is the Stieltjes transform of the asymptotic zero distribution of $\tilde{P}_{\vec{n}}(x)$. Moreover, by solving the equation and by looking for the solution(s) with $zS(z)\to 1$ as $z\to\infty$, we will be able to obtain an explicit expression for it.

\begin{lem} \label{II_alg_eq}
The Stieltjes transform $S(z)$ of $\mu$ satisfies the cubic equation
    $$ z^2 S^3 - z^2 S^2 + q_1 zS + q_2 =0,\quad z\in\C\backslash[0,c].$$
\end{lem}
\begin{proof}
We will start by transforming the differential equation \eqref{II_diff_eq} for $P_{\vec{n}}(z)$ into a differential equation for $\tilde{P}_{\vec{n}}(z)$. Only the asymptotic structure as $n\to\infty$ will be important. By making use of the relation $\tilde{P}^{(k)}_{\vec{n}}(z) = n^k P^{(k)}_{\vec{n}}(nz)$, we find that
    $$  z^2 \tilde{P}'''_{\vec{n}}(z)/n + z(-z+\BO(1)) \tilde{P}''_{\vec{n}}(z) + [z(n_1+\BO(1)) + \BO(1)] \tilde{P}'_{\vec{n}}(z) + n(n_2+\BO(1)) P_{\vec{n}}(z) = 0, $$
as $n\to\infty$. From the identity
    $$ \tilde{P}'_{\vec{n}}(z) = n \tilde{P}_{\vec{n}}(z) S_{\vec{n}}(z) $$
then follows that 
    $$ \tilde{P}''_{\vec{n}}(z) = n^2 \tilde{P}_{\vec{n}}(z) (S_{\vec{n}}(z)^2 + n^{-1} S_{\vec{n}}'(z) )$$
and similarly that
    $$ \tilde{P}'''_{\vec{n}}(z) = n^3 \tilde{P}_{\vec{n}}(z) (S_{\vec{n}}(z)^3 + 3 n^{-1} S_{\vec{n}}(z) S_{\vec{n}}'(z) + n^{-2} S_{\vec{n}}''(z)).$$
By passing to the given subsequence $(n_k)_{k\in\N}$, we may assume that $S_{\vec{n}}$ converges uniformly to $S$ on compact subsets of $\C\backslash[0,c]$ as $n\to\infty$. In that case, the same is true for all the derivatives. We may therefore write
    $$ \tilde{P}^{(k)}_{\vec{n}}(z) = n^{k} \tilde{P}_{\vec{n}}(z) (S_{\vec{n}}^k(z) + \BO(n^{-1})),\quad n\to\infty,$$
which holds uniformly on compact subsets of $\C\backslash[0,c]$. If we plug this in the above equation for $\tilde{P}_{\vec{n}}(z)$, divide by $n^2 \tilde{P}_{\vec{n}}(z)$ and take the limit $n\to\infty$, we obtain the desired algebraic equation for $S(z)$.
\end{proof}

Observe that in the Laguerre case $(q_1,q_2)=(1,0)$, the cubic equation corresponds to the quadratic equation 
$$zS^2-zS+1=0,$$ 
which is, as expected, the algebraic equation for the Stieltjes transform of the asymptotic zero distribution of the scaled Laguerre polynomials $L_n^{(\alpha)}(nx)$ (the Marchenko-Pastur distribution).
\medbreak

\begin{thm} \label{II_AZD}
The density of the asymptotic zero distribution of $P_{\vec{n}}(nx)$ as $n\to\infty$ is given by
	$$ u(x;q_1,q_2) = \frac{1}{2^{4/3}3^{1/2}\pi} \frac{\sqrt[3]{r_{+}(x;q_1,q_2)} - \sqrt[3]{r_{-}(x;q_1,q_2)}}{\sqrt[3]{x^2}}, \quad x\in[0,c], $$
where
	$$ r_{\pm}(x;q_1,q_2) = 2x^2-9q_1x-27q_2 \pm 3\sqrt{3} \sqrt{-(1+q_2)^2x^2+(4q_1^3 + 18q_1q_2)x+27q_2^2} $$
and
    $$ c_{q_1,q_2} = \frac{(2q_1^3 + 9q_1q_2) + 2(q_2^2+q_2+1)^{3/2}}{(1+q_2)^2}.$$
\end{thm}
\begin{proof}
We can solve the cubic equation (\ref{II_alg_eq}) explicitly by making use of Cardano's formula. 
The three solutions are given by 
    $$ S_k(z) = \frac{1}{3} + \zeta^k T_{+}(z)^{1/3} + \bar{\zeta}^{k} T_{-}(z)^{1/3},\quad k=0,1,2, $$
in terms of the third root of unity $\zeta = e^{2\pi i/3}$ and the functions
    $$ T_{\pm}(z) = R(z) \pm (R(z)^2+Q(z)^3)^{1/2},\quad R(z) = \frac{-9q_1/z-27q_2/z^2+2}{54} ,\quad Q(z) = \frac{3q_1/z-1}{9}. $$
The square and cubic root may be taken along any branch cut; swapping to another branch of the cubic root would just permute the solutions $S_k(z)$, while changing the branch of the square root would simply interchange the roles of $T_{\pm}(z)$. 
We are only interested in the solution(s) with $zS(z)\to1$ for $z\to\infty$. By making use of a computer algebra system, we can find
    $$ z S_0(z) \to \infty,\quad z S_1(z) \to \frac{q_1-q_2-1}{2},\quad z S_2(z) \to \frac{q_1+q_2+1}{2} = 1.$$
Consequently, $S_2(z)$ is the Stieltjes transform of the probability measure that we are looking for. We can recover the underlying density $u(x)$ by making use of the Stieltjes-Perron inversion formula
(and $\zeta^2 = -(1+\sqrt{3} i)/2$)
    $$ u(x) = -\lim_{y\to0^{+}} \frac{1}{\pi} \text{Im}(S_2(x+iy)) = \frac{\sqrt{3}}{6\pi} \left(\frac{r_{+}(x)}{2x^2}\right)^{1/3} - \frac{\sqrt{3}}{6\pi} \left(\frac{r_{-}(x)}{2x^2}\right)^{1/3}. $$
Note that the complex roots have become real roots. For the cubic root, this means that $\sqrt[3]{x}=\text{sgn}(x) \sqrt[3]{|x|}$. It is justified to take the real square root, because the function in the argument is positive for $x\in[0,c]$. The constant $c$ corresponds to the largest zero of $u(x)$. The zeros of $u(x)$ are determined by the equation 
$$ -(1+q_2)^2x^2+(4q_1^3 + 18q_1q_2)x+27q_2^2 = 0, $$
which has solutions
    $$ x_{\pm} = \frac{(2q_1^3 + 9q_1q_2)\pm 2(q_2^2+q_2+1)^{3/2}}{(1+q_2)^2}.$$
We have $x_{-}<x_{+}$ and thus $c=x_{+}$.
\end{proof}
\medbreak

In the Laguerre case $(q_1,q_2)=(1,0)$, the above expression becomes
    $$ u(x;1,0) = \frac{1}{2 \pi} \frac{\sqrt{4-x}}{\sqrt{x}},\quad x\in[0,4], $$
which is the expected Marchenko-Pastur distribution. By making use of the basic identity $a-b = (a^3-b^3)/(a^2+ab+b^2)$, the connection between the two distributions can be made more clear:
    $$ u(x;q_1,q_2) = \frac{3}{2^{1/3}\pi} \frac{\sqrt{-(1+q_2)^2x^2+(4q_1^3 + 18q_1q_2)x+27q_2^2}}{\sqrt[3]{x^{2}} (\sqrt[3]{r_{+}(x;q_1,q_2)^2}+\sqrt[3]{r_{-}(x;q_1,q_2)^2} + \sqrt[3]{4x} (x-3q_1) ) } .  $$

In Fig. \ref{fig:zerodensity} we plotted the density $u(x;q_1,q_2)$ for several values of $q_1\in[\frac{1}{2},1]$.

\begin{figure}[h]
\centering
\includegraphics[width=4in]{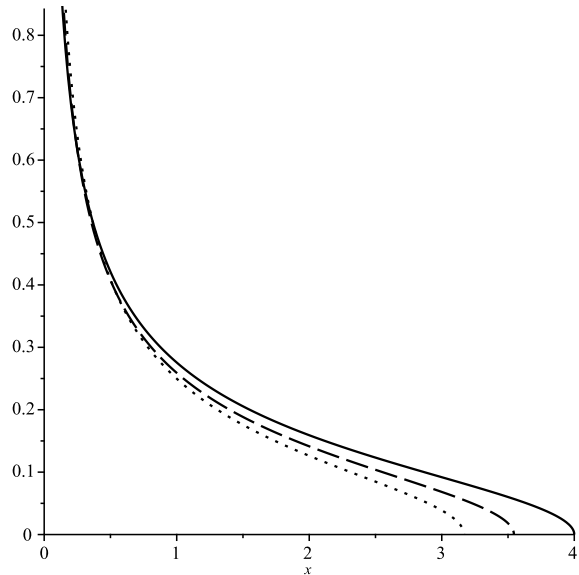}
\caption{The density $u(x;q_1,q_2)$ for $q_1=1$ (solid), $q_1=\frac34$ (dash) and
$q_1=\frac12$ (dots)}
\label{fig:zerodensity}
\end{figure}

The figure shows that as $q_1$ decreases, or equivalently $q_2$ increases, the zeros of the polynomials $P_{\vec{n}}$ are pulled closer towards $0$. Typically for Nikishin systems of two weights, the zeros of the type II polynomials are contained in the support of the first weight and are attracted to the support of the second weight. By increasing the contribution of the second weight, which had support $(-\infty,0]$, we thus observe that this effect becomes stronger.
\medbreak

The behavior of the density $u(x)$ near its endpoints is
    $$ u(x) \sim \frac{\sqrt{3}}{2\pi} q_2^{1/3} x^{-2/3},\quad x\to 0, $$
    $$ u(x) = \BO((c-x)^{1/2}),\quad x\to c. $$
Near the right endpoint, we have the usual square root behavior that also occurs for the asymptotic zero distribution of the Laguerre polynomials and the multiple Laguerre polynomials of the first kind on the step-line. Near the left endpoint, we have the same behavior as the latter but with a different constant.
\bigbreak

We end this section by noting that the asymptotic zero distribution $\nu$ of the scaled type I polynomial $\tilde{B}_{\vec{n}}(x)=B_{\vec{n}}(-nx)$ is the same as the asymptotic zero distribution $\mu$ of $\tilde{P}_{\vec{n}}(x)$, so that the asymptotic zero distribution of $B_{\vec{n}}(nx)$ has support $[-c,0]$ and is the reflection of $\mu$ over $0$. This follows from the fact that the Stieltjes transform $S(z;\nu)$ satisfies the same algebraic equation as $S(z;\mu)$ because the coefficients in the hypergeometric differential equations for $\tilde{B}_{\vec{n}}(x)$ and $\tilde{P}_{\vec{n}}(x)$ are asymptotically equivalent. It also explains the solution $S_1(z)$ in the proof of Theorem \ref{II_AZD}, which had a finite limit $zS(z)$ as $z\to\infty$ as well; the change of variables $z\mapsto -z$ corresponds to letting $q_1\mapsto-q_1$ so that now $-zS_1(-z)\to 1$ as $z\to\infty$.

\subsection{Connection to random matrix theory}

The multiple orthogonal polynomials that were investigated in the previous sections appear naturally in the study of the squared singular values of the product of a truncated Haar distributed random unitary matrix with a Ginibre matrix. A (complex) Ginibre matrix is a random matrix whose entries are independent standard complex Gaussian random variables, see \cite{Ginibre}. There is a similar connection between the multiple orthogonal polynomials associated with the $K$-Bessel functions in \cite{VAYakubovich} and the product of two Ginibre matrices as demonstrated in \cite{KuijlaarsZhang}. In the latter, it is also shown how products of more than two Ginibre matrices are related to multiple orthogonal polynomials associated with certain Meijer-$G$ functions. Products of several Ginibre matrices with a truncated unitary matrix and products that only consist of truncated unitary matrices were investigated in \cite{KuijlaarsStivigny} and \cite{KieburgKuijlaarsStivigny} respectively. A potential connection to multiple orthogonal polynomials however was not mentioned here.
\medbreak

It is known that the joint probability density of the squared singular values of a size $(n+a) \times n$ upper left corner of a random unitary matrix of size $(n+a+b)\times (n+a+b)$ (assuming that $b\geq n$) is given by a polynomial ensemble with probability density
    $$ \propto \Delta_n(x) \det\left[x_j^{a+k-1}(1-x_j)^{b-n}\right]_{j,k=1}^n,\quad \text{ all } x_j\in[0,1]. $$
Here $\Delta_n(x)$ denotes the Vandermonde determinant $\Delta_n(x)=\prod_{1\leq i<j\leq n} (x_j-x_i)$. On the other hand, the joint probability density of the squared singular values of a size $(n+a) \times n$ Ginibre matrix is determined by a polynomial ensemble with probability density
    $$ \propto \Delta_n(x) \det\left[x_j^{a+k-1}e^{-x_j}\right]_{j,k=1}^n, \quad \text{ all } x_j>0.$$
Observe that in these ensembles a beta and gamma density is used respectively. 
\medbreak

Suppose that $n=n_1+n_2$ and consider the product $TG$ of a size $(n+a+b) \times (n+a)$ matrix $T$ that arises by truncating a size $(n+n_2+a+b)\times (n+n_2+a+b)$ random unitary matrix with a Ginibre matrix $G$ of size $(n+a)\times n$. It was proven in \cite{KieburgKuijlaarsStivigny} that the joint probability density of the squared singular values of $TG$ is then given by a polynomial ensemble with probability density
\begin{equation} \label{PE_TG}
    \propto \Delta_n(x) \det\left[w_{k-1}(x_j)\right]_{j,k=1}^n, \quad \text{ all } x_j>0,
\end{equation}
in terms of weights $w_k$ that arise as the Mellin convolution of a beta and gamma density
    $$ w_k(x) = \int_0^1 t^{a+b} (1-t)^{n_2-1} (x/t)^{a+k} e^{-x/t} \frac{dt}{t},\quad x>0.$$
These weights are connected to the exponential integral in the following way
    $$ w_k(x) = x^{a+k} \sum_{j=0}^{n_2-1} (-1)^j \binom{n_2-1}{k} E_{b+j-k}(x) $$
(use the binomial theorem and the change of variables $t\mapsto 1/t$).

Typically one is interested in the correlation kernel
    $$ K_n(x,y) = \sum_{k=0}^{n-1} P_k(x)Q_k(y)$$
of such polynomial ensembles. The kernel may be built out of monic polynomials $P_k$ of degree $k$ and functions $Q_l$ in the linear span of $w_0,\dots,w_{l}$ that satisfy the biorthogonality conditions
\begin{equation*}
    \int_0^\infty P_k(x) Q_l(x) dx = \delta_{k,l},\quad k,l=0,\dots,n-1.
\end{equation*}
These are satisfied under the following conditions
\begin{equation} \label{orth_P}
    \int_0^\infty P_k(x) w_j(x) dx = 0,\quad j=0,\dots,k-1,
\end{equation}
\begin{equation} \label{orth_Q}
    \int_0^\infty Q_l(x) x^j dx = \delta_{l,j},\quad j=0,\dots,l.
\end{equation}    
It is not guaranteed that such a construction is possible, however if it exists, the $P_k$ and $Q_l$ are uniquely determined. Constructions using other types of biorthogonal functions $P_k$ and $Q_l$ are also allowed as long as they span the spaces $\text{span}\{1,\dots,x^{n-1}\}$ and $\text{span}\{w_0,\dots,w_{n-1}\}$ respectively. Particular interest goes to the monic polynomial $P_n$ of degree $n$ that extends the above biorthogonality to $k=n$, because it turns out to be the average characteristic polynomial $ \mathbb{E}[ \prod_{j=1}^n (x-x_j)] $ over the particles in the associated polynomial ensemble. 
\medbreak

The result below shows the connection between the functions that satisfy \eqref{orth_P}--\eqref{orth_Q} and the multiple orthogonal polynomials associated with the exponential integral that were studied in the previous sections.

\begin{prop}
The functions $Q_l$ that satisfy \eqref{orth_Q} are related to the type I functions in \eqref{F} in the following way
    $$Q_l(x) \propto x^{a} F_{l-n_2+1,n_2}^{a,b}(x),\quad 2n_2-2 \leq l \leq n-1.$$
The relation between the polynomials $P_k$ that satisfy \eqref{orth_P} and the type II polynomials in \eqref{II_diff} is
    $$ P_k(x)\propto P_{k-n_2,n_2}^{a,b}(x),\quad 2n_2-1 \leq k \leq n. $$
\end{prop}
\begin{proof}
The weight $w_j$ has Mellin transform
    $$ \int_0^\infty w_j(x) x^{s-1} dx = (n_2-1)! \frac{\Gamma(s+a+j)}{(s+a+b)_{n_2}} .$$
Using the same argument as in the proof of Lemma \ref{I_Mellin_lem}, it can be shown that $Q_l$ has a Mellin transform of the form
    $$ \int_0^\infty Q_l(x) x^{s-1} dx \propto \frac{\Gamma(s+a)}{(s+a+b)_{n_2}} (1-s)_{l}.$$
The first result then follows by comparing this to the Mellin transform \eqref{I_Mellin} of the type~I function $F_{l-n_2+1,n_2}^{a,b}$. Similarly as in \cite[\S 4.2]{KuijlaarsStivigny}, we may derive an integral representation for $P_k$ of the form
    $$ P_k(x) \propto \int_{\Sigma_k} \frac{1}{(-t)_{k+1}} \frac{(t+a+b+1)_{n_2}}{\Gamma(t+a+1)} x^{t} dt, $$
where $\Sigma_k$ is a closed contour in the domain $\text{Re}(t)>-1$ that encircles the interval $[0,k]$ once in the positive direction. We may apply the residue theorem in a similar fashion as in the proof of Proposition \ref{I_2F2} to evaluate this contour integral as a ${}_2F_2$ hypergeometric series. After comparing with the ${}_2F_2$ hypergeometric series \eqref{II_2F2} for $ P_{k-n_2,n_2}^{a,b}(x)$, we obtain the second result. 
\end{proof}

\begin{cor}
The average characteristic polynomial $P_n$ over the particles in the polynomial ensemble \eqref{PE_TG} is given by the monic type II polynomial $P_{n_1,n_2}^{a,b}$.
\end{cor}

\section{Mixed type multiple orthogonal polynomials} \label{SMT}
We will investigate two families of mixed type multiple orthogonal polynomials that involve pairs of weights $(e^{-x},E_{\nu+1}(x))$ and $(x^\alpha,x^\beta)$; the orthogonality conditions will be a mixture of type I and type II orthogonality conditions. A similar consideration also appears in Ap\'ery's construction of rational approximants for $\zeta(3)$, see Beukers \cite{Beukers}. Mixed type multiple orthogonal polynomials also appear in the work of Sorokin \cite{Sor_mixed} and in the study of non-intersecting Brownian motions \cite{DaemsKuijl}. We will assume that the parameters satisfy the conditions $\alpha,\beta,\nu,\alpha+\nu,\beta+\nu>-1$ and $\alpha-\beta\not\in\Z$.

\subsection{First family}
We will first investigate the mixed type multiple orthogonal polynomials with respect to the two pairs of weights $(e^{-x},E_{\nu+1}(x))$ and $(x^{\alpha},x^{\beta})$ on $[0,\infty)$. They will depend on two multi-indices $\vec{n}=(n_1,n_2)$ and $\vec{m}=(m_1,m_2)$ that will determine the degrees of the polynomials and the number of orthogonality conditions respectively. The multi-indices are connected through the relation $\left|\vec{n}\right|-1=\left|\vec{m}\right|$. We are looking for polynomials $A_{\vec{n},\vec{m}}$ and $B_{\vec{n},\vec{m}}$ with $\deg A_{\vec{n},\vec{m}}\leq n_1-1$ and $\deg B_{\vec{n},\vec{m}}\leq n_2-1$, for which the function
\begin{equation} \label{defF_MT}
    F_{\vec{n},\vec{m}}(x) = A_{\vec{n},\vec{m}}(x) e^{-x} + B_{\vec{n},\vec{m}}(x) E_{\nu+1}(x)
\end{equation}
satisfies the orthogonality conditions
\begin{align}
     \int_0^\infty F_{\vec{n},\vec{m}}(x) x^{k+\alpha} \, dx &= 0, \quad k=0,\dots,m_1-1,   \label{Mortho1} \\
     \int_0^\infty F_{\vec{n},\vec{m}}(x) x^{k+\beta} \, dx &= 0, \quad k=0,\dots,m_2-1.   \label{Mortho2}
\end{align}     
The polynomials $A_{\vec{n},\vec{m}}$ and $B_{\vec{n},\vec{m}}$ have $\left|\vec{n}\right|$ coefficients and the system of equations \eqref{Mortho1}--\eqref{Mortho2} corresponds to a homogeneous linear system of $\left|\vec{m}\right|$ equations. Due to the condition $\left|\vec{m}\right|=\left|\vec{n}\right|-1$, existence of a non-zero solution is guaranteed. If the matrix of the system has maximal rank, the coefficients of $A_{\vec{n},\vec{m}}$ and $B_{\vec{n},\vec{m}}$ are uniquely determined up to a scalar multiplication.
\medbreak

The (degenerate) case $m_2=0$ corresponds to the setting of the type I polynomials as before. As a consequence, we will encounter some similar properties. On the other hand, the (degenerate) case $n_2=0$ corresponds to the setting of the type II multiple Laguerre polynomial of the first kind. This suggests that, instead of a Laguerre polynomial $L_{n}^{(\alpha)}(x)$, now a type II multiple Laguerre polynomial of the first kind $L_{\vec{n}}^{(II\mid\alpha,\beta)}(x)$ may play a role. Explicit formulas for these polynomials are well-known; a Rodrigues formula can for example be found in \cite[\S 23.4.1]{Ismail}
\begin{equation} \label{LaguerreII_Rodr}
    L_{\vec{n}}^{(II\mid\alpha,\beta)}(x) = (-1)^{\left|\vec{n}\right|} e^{x} x^{-\alpha} \frac{d^{n_1}}{dx^{n_1}}\left[x^{n_1+\alpha-\beta}\frac{d^{n_2}}{dx^{n_2}}\left[x^{n_2+\beta}e^{-x}\right]\right]  .
\end{equation}

In what follows, we will assume that the multi-index $\vec{n}$ satisfies $n_1+1\geq n_2$. This is the natural extension of the restriction we had before; the case $m_2=0$ should correspond to the type I setting.

\subsubsection{Convolution and Rodrigues formula} 
Our approach will be similar as in Section \ref{I_S1}.

\begin{lem}
There is a mixed type function $F_{\vec{n},\vec{m}}$ with Mellin transform
\begin{equation} \label{mixedI_Mellin}
    \hat{F}_{\vec{n},\vec{m}}(s) = \frac{\Gamma(s)}{(s+\nu)_{n_2}} (\alpha+1-s)_{m_1} (\beta+1-s)_{m_2},\quad s>\max\{0,-\nu\}.
\end{equation}
\end{lem}
\begin{proof}
The Mellin transform of a mixed type function may be computed using the same approach as in the proof of Lemma \ref{I_Mellin_lem}. In this way, we can obtain
    $$ \hat{F}_{\vec{n},\vec{m}}(s) = \frac{\Gamma(s)}{(s+\nu)_{n_2}} \left( p_{n_1-1}(s)(s+\nu)_{n_2} + q_{2n_2-2}(s) \right) $$
in terms of polynomials with their subscript as maximal degree. The degree of the polynomial between the parentheses is less than or equal to $\max\{n_1+n_2-1,2n_2-2\}$, which is equal to $n_1+n_2-1=m_1+m_2$ under the assumption that $n_1\geq n_2-1$. Since the orthogonality conditions imply that $\hat{F}_{\vec{n},\vec{m}}(s)$ has zeros at the $m_1+m_2$ values $s=\alpha+1,\dots,\alpha+m_1$ and $s=\beta+1,\dots,\beta+m_2$, the remaining polynomial must be a non-zero scalar multiple of $(\alpha+1-s)_{m_1} (\beta+1-s)_{m_2}$.
\end{proof}

We will compare this Mellin transform to the Mellin transform of a type II Laguerre polynomial of the first kind multiplied by a gamma density, which we will compute below.

\begin{lem} \label{LaguerreII_transf}
The Mellin transform of $L_{\vec{n}}^{(II\mid\alpha,\beta)}(x)e^{-x}$ is given by
$$\int_0^\infty L_{\vec{n}}^{(II\mid\alpha,\beta)}(x) e^{-x} x^s dx = (-1)^{\left|\vec{n}\right|}\Gamma(s)(\alpha+1-s)_{n_1}(\beta+1-s)_{n_2},\quad s>0.$$
\end{lem} 
\begin{proof}
We may use the Rodrigues formula \eqref{LaguerreII_Rodr} and perform integration by parts $n_1$ times to get
$$\begin{aligned}
  \int_0^\infty  L_{\vec{n}}^{(II\mid\alpha,\beta)}(x) e^{-x} x^{s-1} dx  &= (-1)^{\left|\vec{n}\right|+n_1} \int_0^\infty \left(x^{n_1+\alpha-\beta}\frac{d^{n_2}}{dx^{n_2}} x^{n_2+\beta} e^{-x}\right) \left(\frac{d^{n_1}}{dx^{n_1}} x^{s-\alpha-1}\right) dx  \\
  &= (-1)^{\left|\vec{n}\right|} (\alpha+1-s)_{n_1} \int_0^\infty \left(\frac{d^{n_2}}{dx^{n_2}} x^{n_2+\beta} e^{-x}\right) x^{-\beta+s-1} dx .
\end{aligned}$$
The remaining integral is the Mellin transform of $n_2! L_{n_2}^{(\beta)}(x)e^{-x}$ and was computed in Lemma \ref{Laguerre_transf}.
\end{proof}

The previous result allows us to write the mixed type function as a Mellin convolution in which a type II Laguerre polynomial of the first kind appears.

\begin{thm}
The mixed type function with Mellin transform \eqref{mixedI_Mellin} is given by
\begin{equation} \label{mixedI_conv}
    F_{\vec{n},\vec{m}}(x) = \int_{x}^{\infty} L_{\vec{m}}^{(II\mid\alpha,\beta)}(t) e^{-t} \left(1-\frac{x}{t}\right)^{n_2-1} \left(\frac{x}{t}\right)^{\nu} \, \frac{dt}{t}
\end{equation}
multiplied by $(-1)^{\left|\vec{n}\right|}/(n_2-1)!$.
\end{thm}
\begin{proof}
The Mellin transform \eqref{mixedI_Mellin} is $(-1)^{\left|\vec{n}\right|}/(n_2-1)!$ times the product of the Mellin transform of a type II Laguerre polynomial multiplied with a gamma density
$$\int_0^\infty L_{\vec{m}}^{(II\mid\alpha,\beta)}(x) e^{-x} x^{s-1}\, dx = (-1)^{\left|\vec{n}\right|} \Gamma(s) (\alpha+1-s)_{m_1} (\beta+1-s)_{m_2},$$
(see Lemma \ref{LaguerreII_transf}) and of a beta density
$$ \int_0^{\infty} (1-x)^{n_2-1} x^{s+\nu-1} \chi_{[0,1]}(x) \,dx = B(s+\nu,n_2) = \frac{(n_2-1)!}{(s+\nu)_{n_2}}.$$
Hence we can use the multiplicative property of the Mellin convolution and the uniqueness of the Mellin transform to obtain the above.
\end{proof}

The structure of the mixed type function in \eqref{mixedI_conv} resembles the function that is used to generate the rational approximants in Apéry's proof of the irrationality of $\zeta(3)$. There one uses a Mellin convolution of a (multiple) Jacobi-Piñeiro polynomial and a beta density, see \cite{SmetVA}. Instead of a Jacobi-Piñeiro polynomial, we use a multiple Laguerre polynomial multiplied with a gamma density.
\medbreak

The convolution formula allows us to derive a Rodrigues formula for the mixed type function.

\begin{thm}
The mixed type function in \eqref{mixedI_conv} is given by the Rodrigues formula
    $$ F_{\vec{n},\vec{m}}(x) = \frac{(-1)^{\left|\vec{n}\right|}}{(\nu+1)_{n_2-1}} x^{-\alpha} \frac{d^{m_1}}{dx^{m_1}}\left[x^{m_1+\alpha-\beta}\frac{d^{m_2}}{dx^{m_2}}\left[x^{m_2+\beta} e^{-x} \frac{d^{n_2-1}}{dx^{n_2-1}}\left[x^{n_2-1} e^{x} E_{\nu+1}(x) \right] \right] \right].$$
\end{thm}
\begin{proof}
The proof will be similar to the proof of the Rodrigues formula \eqref{I_Rodr} for the type~I function; now we will use the Rodrigues formula for the underlying multiple Laguerre polynomial. We again start with the change of variables $t\mapsto xt$ in order to put the dependence on $x$ in the multiple Laguerre polynomial 
    $$ F_{\vec{n},\vec{m}}(x) = \int_1^\infty L_{\vec{m}}^{(II\mid\alpha,\beta)}(xt) e^{-xt} \left(1-\frac{1}{t}\right)^{n_2-1} \, \frac{dt}{t^{\nu+1}}. $$
The Rodrigues formula \eqref{Laguerre_Rodr} for $L_{\vec{m}}^{(II\mid\alpha,\beta)}(x)$ and the chain rule then imply that
    $$ L_{\vec{m}}^{(II\mid\alpha,\beta)}(xt) e^{-xt} = (-1)^{\left|\vec{n}\right|} (xt)^{-\alpha} \frac{d^{m_1}}{dx^{m_1}}\left[(xt)^{m_1+\alpha-\beta}\frac{d^{m_2}}{dx^{m_2}}\left[(xt)^{m_2+\beta} e^{-xt} \right] \cdot t^{-m_2} \right] \cdot t^{-m_1},$$
or after rearranging some factors depending on $t$,
$$ L_{\vec{m}}^{(II\mid\alpha,\beta)}(xt) e^{-xt} = (-1)^{\left|\vec{n}\right|} x^{-\alpha} \frac{d^{m_1}}{dx^{m_1}}\left[x^{m_1+\alpha-\beta}\frac{d^{m_2}}{dx^{m_2}}\left[x^{m_2+\beta} e^{-xt} \right] \right].$$
We can then take all the derivatives outside of the integral to get
    $$ F_{\vec{n},\vec{m}}(x) = (-1)^{\left|\vec{n}\right|} x^{-\alpha} \frac{d^{m_1}}{dx^{m_1}}\left[x^{m_1+\alpha-\beta}\frac{d^{m_2}}{dx^{m_2}}\left[x^{m_2+\beta} \int_1^\infty e^{-xt} \left(1-\frac{1}{t}\right)^{n_2-1} \, \frac{dt}{t^{\nu+1}} \right] \right].$$
The remaining integral can be written in the desired way using the same argument as in the proof of formula \eqref{I_Rodr}.
\end{proof}

\subsubsection{Hypergeometric series}

Similarly as with the type I function, we can derive a representation of the mixed type function in terms of certain generalized hypergeometric series.

\begin{prop}
Suppose that $\nu\not\in\Z$. The mixed type function with Mellin transform \eqref{mixedI_Mellin} arises as the following linear combination of ${}_3F_3$ hypergeometric series
$$\begin{aligned}
    F_{\vec{n},\vec{m}}(x) =\, & \mathcal{A}_{\vec{n},\vec{m}} \ {}_3F_3 \left( \begin{array}{c} m_1+\alpha+1,m_2+\beta+1,-n_2-\nu+1 \\ \alpha+1,\beta+1,-\nu+1 \end{array}; -x \right) \\
    &+ \mathcal{B}_{\vec{n},\vec{m}} \ {}_3F_3 \left( \begin{array}{c} -n_2+1,m_1+\alpha+\nu+1,m_2+\beta+\nu+1 \\ \nu+1,\alpha+\nu+1,\beta+\nu+1 \end{array}; -x \right) x^\nu,
\end{aligned}$$
where
    $$ \mathcal{A}_{\vec{n},\vec{m}} = \frac{(\alpha+1)_{m_1} (\beta+1)_{m_2}}{(\nu)_{n_2}},\quad \mathcal{B}_{\vec{n},\vec{m}} = \frac{(\alpha+\nu+1)_{m_1} (\beta+\nu+1)_{m_2}\Gamma(-\nu)}{(n_2-1)!}. $$
\end{prop}
\begin{proof}
We will use the same strategy as in the proof of Theorem \ref{I_2F2} and take a closer look at the Mellin transform of $x^\alpha F_{\vec{n},\vec{m}}(x)$, which is
$$\int_0^\infty F_{\vec{n},\vec{m}}(x) x^{\alpha+s-1} dx = \frac{\Gamma(s+\alpha)}{(s+\alpha+\nu)_{n_2}} (1-s)_{m_1} (\beta-\alpha+1-s)_{m_2},\quad s>-\min\{\alpha,\alpha+\nu\},$$
(starting from $x^{\beta} F_{\vec{n},\vec{m}}(x)$ will lead to the same result). We can extend this function to a function that is analytic in the complex plane except for the simple poles $-\alpha,-\alpha-1,\dots$ and $-\alpha-\nu,\dots,-\alpha-\nu-n_2+1$. All these poles lie in an interval $(-\infty,c)$ for some $c\in\R$ with $-\min\{\alpha,\alpha+\nu,0\}<c<1$. The original function can then be recovered by evaluating the inverse Mellin transform via the residue theorem
$$\begin{aligned}
    x^{\alpha} F_{\vec{n},\vec{m}}(x) &= \int_{c-i\infty}^{c+\infty} \frac{\Gamma(s+\alpha)}{(s+\alpha+\nu)_{n_2}} (1-s)_{m_1} (\beta-\alpha+1-s)_{m_2} x^{-s} ds \\
    &= \sum_{k=0}^{\infty} \frac{(-1)^k}{k!} \frac{(k+\alpha+1)_{m_1} (k+\beta+1)_{m_2}}{(-k+\nu)_{n_2}} x^{\alpha+k} \\
    & \quad + \sum_{k=0}^{n_2-1} \frac{\Gamma(-k-\nu)}{(-k)_k (1)_{n_2-k}} (k+\alpha+\nu+1)_{m_1} (k+\beta+\nu+1)_{m_2} x^{\alpha+\nu+k}.
\end{aligned}$$
After some elementary computations, we obtain the stated result.  
\end{proof}

The two terms in the above expression are again solutions of the same hypergeometric differential equation (now of the fourth order). The two other solutions are
    $$ x^{-\alpha} {}_3F_3 \left( \begin{array}{c} -n_2-\alpha-\nu+1,m_1+1,m_2+\beta-\alpha+1 \\ -\alpha-\nu+1,-\alpha+1,\beta-\alpha+1 \end{array}; -x \right), $$
    $$ x^{-\beta} {}_3F_3 \left( \begin{array}{c} -n_2-\alpha-\nu+1,m_1+\alpha-\beta+1,m_2+1\\ -\beta-\nu+1,\alpha-\beta+1,-\beta+1 \end{array}; -x \right). $$
Using the same argument as in Proposition \ref{IB_2F2}, we then must have
$$B_{\vec{n},\vec{m}}(x) = B_{\vec{n},\vec{m}}(0) \cdot {}_3F_3 \left( \begin{array}{c} -n_2+1,m_1+\alpha+\nu+1,m_2+\beta+\nu+1 \\ \nu+1,\alpha+\nu+1,\beta+\nu+1 \end{array}; -x \right).$$

\subsection{Second family}
In this section, we will investigate the mixed type multiple orthogonal polynomials with respect to the two pairs of weights $(x^{\alpha},x^{\beta})$ and $(e^{-x},E_{\nu+1}(x))$ on $[0,\infty)$. By reversing the order of the weights, we end up in a setting that is dual to the one before. Now, we are looking for polynomials $C_{\vec{n},\vec{m}}$ and $D_{\vec{n},\vec{m}}$ with $\deg C_{\vec{n},\vec{m}}\leq n_1-1$ and $\deg D_{\vec{n},\vec{m}}\leq n_2-1$, for which the function
\begin{equation} \label{defG_MT}
    G_{\vec{n},\vec{m}}(x) = C_{\vec{n},\vec{m}}(x) x^{\alpha} + D_{\vec{n},\vec{m}}(x) x^{\beta}
\end{equation}
satisfies the orthogonality conditions
$$\begin{aligned}
     \int_0^\infty G_{\vec{n},\vec{m}}(x) x^{k} e^{-x} \, dx &= 0, \qquad k=0,\dots,m_1-1, \\
     \int_0^\infty G_{\vec{n},\vec{m}}(x) x^{k} E_{\nu+1}(x) \, dx &= 0, \qquad k=0,\dots,m_2-1.
\end{aligned}$$   
\medbreak
The (degenerate) case $n_2=0$ corresponds to the setting of the type II polynomials of before. This will again lead to some similar properties. The (degenerate) case $m_2=0$ corresponds to the setting of the type I multiple Laguerre polynomials of the first kind. The associated type I functions will be denoted by $L_{n,m}^{(I\mid\alpha,\beta)}(x)$ and are of the form
\begin{equation} \label{defLI}
    L_{n,m}^{(I\mid\alpha,\beta)}(x) = C_{n,m}(x) x^{\alpha} + D_{n,m}(x) x^{\beta},
\end{equation}
where $C_{n,m}$ and $D_{n,m}$ are polynomials with $\deg C_{n,m}=n-1$ and $\deg D_{n,m}=m-1$. They satisfy the orthogonality conditions
     $$\int_0^\infty L_{n,m}^{(I,\alpha,\beta)}(x) x^{k} e^{-x} \, dx = 0, \quad k=0,\dots,n+m-2 , $$
and are normalized such that
    $$\int_0^\infty L_{n,m}^{(I,\alpha,\beta)}(x) x^{n+m-1} e^{-x} \, dx = 1.$$
We don't believe any explicit formulas for the type I multiple Laguerre polynomials of the first kind have been written down before, so we will handle that case first.

In the derivation of the mixed type polynomials, we will assume that the multi-index $\vec{m}$ satisfies $m_1+1\geq m_2$. This is the natural extension of the restriction we had before.

\subsubsection{Type I multiple Laguerre polynomials of the first kind}

\begin{prop}
The polynomials $C_{n,m}(x)$ and $D_{n,m}(x)$ in \eqref{defLI} are given by
\begin{align}
    C_{n+1,m+1}(x)= \frac{1}{Z_{n,m}} \sum_{k=0}^{n} \binom{n+\alpha-\beta}{n-k}\binom{m-\alpha+\beta}{k} \frac{x^k}{\Gamma(k+\alpha+1)}, \label{MLaguerreI_1} \\
    D_{n+1,m+1}(x)= - \frac{1}{Z_{n,m}} \sum_{k=0}^{m} \binom{m-\alpha+\beta}{m-k}\binom{n+\alpha-\beta}{k} \frac{x^k}{\Gamma(k+\beta+1)}, \label{MLaguerreI_2}
\end{align}
where $Z_{n,m}=(-1)^{n+m}(\alpha-\beta)(\alpha-\beta+1)_n(\beta-\alpha+1)_m$ is a normalizing constant. In particular, $D_{n,m}^{\alpha,\beta}(x) = C_{m,n}^{\beta,\alpha}(x)$.
\end{prop}
\begin{proof}
We will prove this inductively on the sum $n+m$. For $n+m$ equal to $0$ or $1$, we can verify the formula by hand or by means of a computer algebra system. Now, let $(n,m)$ be a multi-index with $n+m\geq 1$ and suppose that the formulas \eqref{MLaguerreI_1}--\eqref{MLaguerreI_2} hold for all multi-indices that add up to a smaller amount. We will prove that they are also valid for $(n,m)$ by making use of the two nearest neighbor recurrence relations

$$ xL_{n,m}^{(I\mid\alpha,\beta)}(x) = L_{n-1,m}^{(I\mid\alpha,\beta)}(x)  + c_{n-1,m} L_{n,m}^{(I\mid\alpha,\beta)}(x) + a_{n,m} L_{n+1,m}^{(I\mid\alpha,\beta)}(x) + b_{n,m}L_{n,m+1}^{(I\mid\alpha,\beta)}(x), $$
$$ xL_{n,m}^{(I\mid\alpha,\beta)}(x) = L_{n,m-1}^{(I\mid\alpha,\beta)}(x) + d_{n,m-1} L_{n,m}^{(I\mid\alpha,\beta)}(x) +  a_{n,m} L_{n+1,m}^{(I\mid\alpha,\beta)}(x) + b_{n,m}L_{n,m+1}^{(I\mid\alpha,\beta)}(x). $$
If we subtract them, we find another (simpler) recurrence relation
    $$ (d_{n,m-1}-c_{n-1,m}) L_{n,m}^{(I\mid\alpha,\beta)}(x) = L_{n-1,m}^{(I\mid\alpha,\beta)}(x) - L_{n,m-1}^{(I\mid\alpha,\beta)}(x).$$
We know that the corresponding type II multiple orthogonal polynomials $L_{n,m}^{(II\mid\alpha,\beta)}$, which are multiple Laguerre polynomials of the first kind, satisfy nearest neighbor recurrence relations with similar coefficients
$$ xL_{n,m}^{(II\mid\alpha,\beta)}(x) = L_{n+1,m}^{(II\mid\alpha,\beta)}(x)  + c_{n,m} L_{n,m}^{(II\mid\alpha,\beta)}(x) + a_{n,m} L_{n-1,m}^{(II\mid\alpha,\beta)}(x) + b_{n,m} L_{n,m-1}^{(II\mid\alpha,\beta)}(x), $$
$$ xL_{n,m}^{(II\mid\alpha,\beta)}(x) = L_{n,m-1}^{(II\mid\alpha,\beta)}(x) + d_{n,m} L_{n,m}^{(II\mid\alpha,\beta)}(x) +  a_{n,m} L_{n-1,m}^{(II\mid\alpha,\beta)}(x) + b_{n,m} L_{n,m-1}^{(II\mid\alpha,\beta)}(x). $$
The coefficients in the nearest neighbor recurrence relations for the multiple Laguerre polynomials of the first kind were determined in \cite{WVA}. In particular,
    $$c_{n,m} = 2n+m+\alpha+1,\quad d_{n,m} = n+2m+\beta+1. $$
Let us denote the coefficient of the monomial $x^k$ in a polynomial $P(x)$ by $P[k]$. If we compare the coefficient of $x^k$ in the $C$-polynomial on each side of the equation
$$ (-n+m-\alpha+\beta) L_{n,m}^{(I\mid\alpha,\beta)}(x) = L_{n-1,m}^{(I\mid\alpha,\beta)}(x) - L_{n,m-1}^{(I\mid\alpha,\beta)}(x),$$
we get
$$ (-n+m-\alpha+\beta) C_{n,m}[k] = C_{n-1,m}[k] - C_{n,m-1}[k]. $$
We may compute this difference via the induction hypothesis; it is equal to
\begin{align*}
& \frac{1}{Z_{n-1,m}} \binom{n-1+\alpha-\beta}{n-1-k}\binom{m-\alpha+\beta}{k}  -\frac{1}{Z_{n,m-1}} \binom{n+\alpha-\beta}{n-k}\binom{m-1-\alpha+\beta}{k}  \\
& = \frac{1}{Z_{n,m}} \binom{n+\alpha-\beta}{n-k}\binom{m-\alpha+\beta}{k} \left( \frac{Z_{n,m}}{Z_{n-1,m}}\frac{n-k}{n+\alpha-\beta} - \frac{Z_{n,m}}{Z_{n,m-1}} \frac{m-k-\alpha+\beta}{m-\alpha+\beta}\right) \\
& = \frac{1}{Z_{n,m}} \binom{n+\alpha-\beta}{n-k}\binom{m-\alpha+\beta}{k} (-n+m-\alpha+\beta).
\end{align*}
This is precisely what we wanted to show for $C_{n,m}[k]$. We have therefore proven \eqref{MLaguerreI_1}. The formula for $D_{n,m}(x)$ in \eqref{MLaguerreI_2} then follows from the symmetry of the problem.
\end{proof}

Some elementary computations show that the type I Laguerre polynomials of the first kind are the following hypergeometric polynomials
    $$ C_{n,m}(x) = \frac{(-1)^{n+m+1}}{\Gamma(\alpha+1) (\beta-\alpha)_{m} (n-1)!} {}_2F_2 \left( \begin{array}{c} -n+1,-m+\alpha-\beta+1 \\ \alpha+1,\alpha-\beta+1 \end{array}; x \right), $$
    $$ D_{n,m}(x) = \frac{(-1)^{n+m+1}}{\Gamma(\beta+1) (\alpha-\beta)_{n} (m-1)!} {}_2F_2 \left( \begin{array}{c} -m+1,-n-\alpha+\beta+1 \\ \beta+1,-\alpha+\beta+1 \end{array}; x \right). $$

\subsubsection{Differentiation formula}

We will be able to determine an explicit expression for the mixed type polynomials by making use of the lemma below, which is the analogue of the reduction property in Lemma \ref{II_red}. In this lemma, we explicitly wrote down the dependence of $G_{\vec{n},\vec{m}}=G_{\vec{n},\vec{m}}^{\alpha,\beta,\nu}$ on the underlying parameters.

\begin{lem}
One has,
    $$ \frac{d}{dx}\left[ x^{\nu} G_{\vec{n},\vec{m}}^{\alpha,\beta,\nu}(x) \right] \propto x^{\nu-1} G_{\vec{n},(m_1-1,m_2+1)}^{\alpha,\beta,\nu-1}(x) . $$
\end{lem}
\begin{proof}
We will first verify that the left hand side has the right structure; it is given by
    $$ ((\alpha+\nu) C_{\vec{n},\vec{m}}(x) + x C_{\vec{n},\vec{m}}'(x)) x^{\alpha+\nu-1}
    + ((\beta+\nu) D_{\vec{n},\vec{m}}(x) + x D_{\vec{n},\vec{m}}'(x)) x^{\beta+\nu-1}, $$
where the polynomials between the brackets are still of degree $n_1$ and $n_2$ respectively. It then remains to show that the required orthogonality conditions are satisfied, i.e.,
    $$ \int_0^\infty \frac{d}{dx}\left[ x^{\nu} G_{\vec{n},\vec{m}}^{\alpha,\beta,\nu}(x) \right] e^{-x} x^{k-\nu+1}\, dx = 0,\quad k=0,\dots,m_1-2,  $$
    $$ \int_0^\infty \frac{d}{dx}\left[ x^{\nu} G_{\vec{n},\vec{m}}^{\alpha,\beta,\nu}(x) \right] E_{\nu}(x) x^{k-\nu+1}\, dx = 0,\quad k=0,\dots,m_2. $$
We can work out both integrals via integration by parts. The first one becomes
$$ -\int_0^\infty G_{\vec{n},\vec{m}}^{\alpha,\beta,\nu}(x) e^{-x} x^{k} (-x+(k-\nu+1))\, dx, $$
and indeed vanishes for $k=0,\dots,m_1-2$. Using the same argument as in the proof of Lemma \ref{II_red}, the second integral becomes
$$ \int_0^\infty G_{\vec{n},\vec{m}}^{\alpha,\beta,\nu}(x) x^{k-1} ((x-k)e^{-x} 
 + k \nu E_{\nu+1}(x)) \, dx.$$
For $k=0$, this is
    $$ \int_0^\infty G_{\vec{n},\vec{m}}(x) e^{-x}\, dx = 0, $$
while for $k>0$, the integral vanishes if $k\leq m_1-1$ and $k\leq m_2$. We thus also obtain the desired orthogonality conditions with respect to the second weight, since $m_1\geq m_2+1$ by assumption.
\end{proof}

We can now express the mixed type function in terms of the type I function of the type~I multiple Laguerre polynomials of the first kind.

\begin{thm}
There is a mixed type function $G_{\vec{n},\vec{m}}$ given by
\begin{equation} \label{mixedII_form}
    G_{\vec{n},\vec{m}}(x) = x^{-\nu}\frac{d^{m_2}}{dx^{m_2}}\left[ x^{m_2+\nu} L_{\vec{n}}^{(I\mid\alpha,\beta)}(x) \right].
\end{equation}
\end{thm}
\begin{proof}
It follows from the previous lemma that
    $$ G_{\vec{n},\vec{m}}^{\alpha,\beta,\nu}(x) \propto x^{-\nu} \frac{d}{dx}\left[ x^{\nu+1} G_{\vec{n},(m_1+1,m_2-1)}^{\alpha,\beta,\nu+1}(x) \right], $$
for $m_1+1\geq m_2$. After extracting $m_2-1$ additional derivatives, we thus get
    $$ G_{\vec{n},\vec{m}}^{\alpha,\beta,\nu}(x) \propto x^{-\nu} \frac{d^{m_2}}{dx^{m_2}}\left[ x^{m_2+\nu} G_{\vec{n},(m_1+m_2,0)}^{\alpha,\beta,\nu+m_2}(x) \right]. $$
The multi-index $(m_1+m_2,0)$ corresponds to the type I setting with $m_1+m_2=n_1+n_2-1$ orthogonality conditions and therefore $G_{\vec{n},(m_1+m_2,0)}^{\alpha,\beta,\nu+m_2}(x) \propto L_{\vec{n}}^{(I\mid\alpha,\beta)}(x)$.
\end{proof}

\subsubsection{Hypergeometric series}

The mixed type polynomials can be represented as ${}_3F_3$ hypergeometric polynomials in the following way.

\begin{prop}
The mixed type polynomials $C_{\vec{n},\vec{m}}(x)$ and $D_{\vec{n},\vec{m}}(x)$ in the mixed type function \eqref{mixedII_form} arise as the hypergeometric series
    $$ C_{\vec{n},\vec{m}}(x) = \mathcal{C}_{\vec{n},\vec{m}} \, {}_3F_3 \left( \begin{array}{c} -n_1+1,-n_2+\alpha-\beta+1,m_2+\alpha+\nu+1 \\ \alpha+1,\alpha-\beta+1,\alpha+\nu+1 \end{array}; x \right) ,$$
    $$ D_{\vec{n},\vec{m}}(x) = \mathcal{D}_{\vec{n},\vec{m}} \, {}_3F_3 \left( \begin{array}{c} -n_2+1,-n_1-\alpha+\beta+1,m_2+\beta+\nu+1 \\ \beta+1,-\alpha+\beta+1,\beta+\nu+1 \end{array}; x \right) ,$$
where
$$ \mathcal{C}_{\vec{n},\vec{m}} = \frac{(-1)^{n_1+n_2+1} (\alpha+\nu+1)_{m_2}}{\Gamma(\alpha+1) (\beta-\alpha)_{n_2} (n_1-1)!},\quad \mathcal{D}_{\vec{n},\vec{m}} = \frac{(-1)^{n_1+n_2+1} (\beta+\nu+1)_{m_2}}{\Gamma(\beta+1)(\alpha-\beta)_{n_1} (n_2-1)! }. $$
\end{prop}
\begin{proof}
This follows from the hypergeometric series representations for the type I multiple Laguerre polynomials of the first kind and the differentiation properties of the generalized hypergeometric series.
\end{proof}

\section{Related multiple orthogonal polynomials}

The type I and II multiple orthogonal polynomials that were investigated here are related to the two-parameter subfamily of the Jacobi-Piñeiro polynomials in, e.g., \cite{SmetVA}, and the multiple orthogonal polynomials associated with $K$-Bessel functions in \cite{VAYakubovich}, with confluent hypergeometric functions in \cite{LimaLoureiro1} and with Gauss' hypergeometric functions in \cite{LimaLoureiro2}. In the latter, one considers multiple orthogonal polynomials on the step-line associated with two weights $\mathcal{W}(x;a,b,c,d)$ and $\mathcal{W}(x;a,b+1,c+1,d)$. The weights are defined in terms of Gauss’ hypergeometric function and have moments of the form
    $$ \int_0^\infty \mathcal{W}(x;a,b,c,d) x^k dx = \frac{(a+1)_k(b+1)_k}{(c+1)_k(d+1)_k}.$$
One assumes that $a,b,c,d\in\R_{>-1}$, $c+1,d>a$ and $c,d>b$ for regularity reasons.

The connection to the multiple orthogonal polynomials in this article is given in Figure \ref{fig1} below. It is based on the formulas for the Mellin transforms of the type I functions (where we might have had to rebalance the underlying weights) as given in Figure \ref{fig2}. 

\begin{figure}[ht]
    \centering
    \includegraphics[scale=0.33]{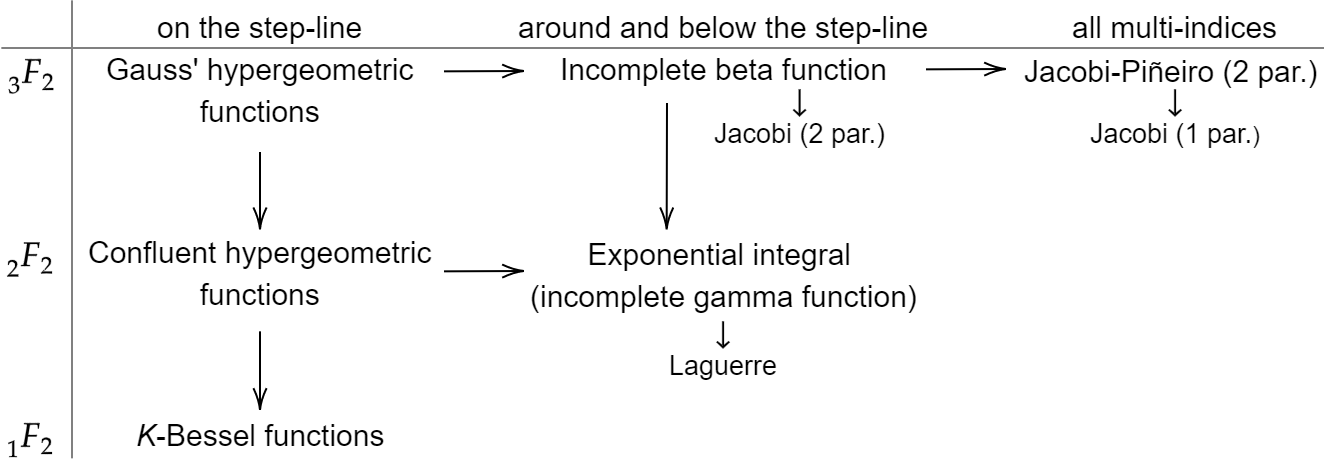}
    \caption{Connection to other weights.}
    \label{fig1}
\end{figure}

The vertical order is according to the type of hypergeometric function that is used in the expression for the type I functions (and the type II polynomials) which can be obtained by applying the inverse Mellin transform. Moreover, the asymptotic zero distribution of the type II multiple orthogonal polynomials on the step-line is the same on each row. Moving horizontally in the figure, one obtains formulas that are valid for a larger amount of multi-indices. Classical orthogonal polynomials appear first as degenerate cases of the multiple orthogonal polynomials in the second column.
\newpage

Figure \ref{fig2} suggests that in a setting with a weight that has moments of the form 
    $$ \frac{(a_1+1)_k\dots(a_r+1)_k}{(b_1+1)_k\dots(b_r+1)_k},$$
the Mellin transforms of the type I functions on the step-line will be given by
    $$ \frac{\Gamma(s+a_1)\dots\Gamma(s+a_r)}{\Gamma(s+b_1+n_1)\dots\Gamma(s+b_r+n_r)} (1-s)_{n_1+\dots n_r-1} $$
and that we have to consider $r$ weights in total to attain this. For an appropriate choice of weights, it can be verified that this is indeed the case. In that case, the type I function arises as an $r$-fold Mellin convolution in which Jacobi polynomials and beta densities appear. An application of the inverse Mellin transform then allows us to obtain a formula for the type I function in terms of $r$ hypergeometric functions of type $_{r+1}F_r$. From this, we may derive expressions for the type II polynomials as well.

\begin{figure}[h]
    \centering
    \includegraphics[scale=0.33]{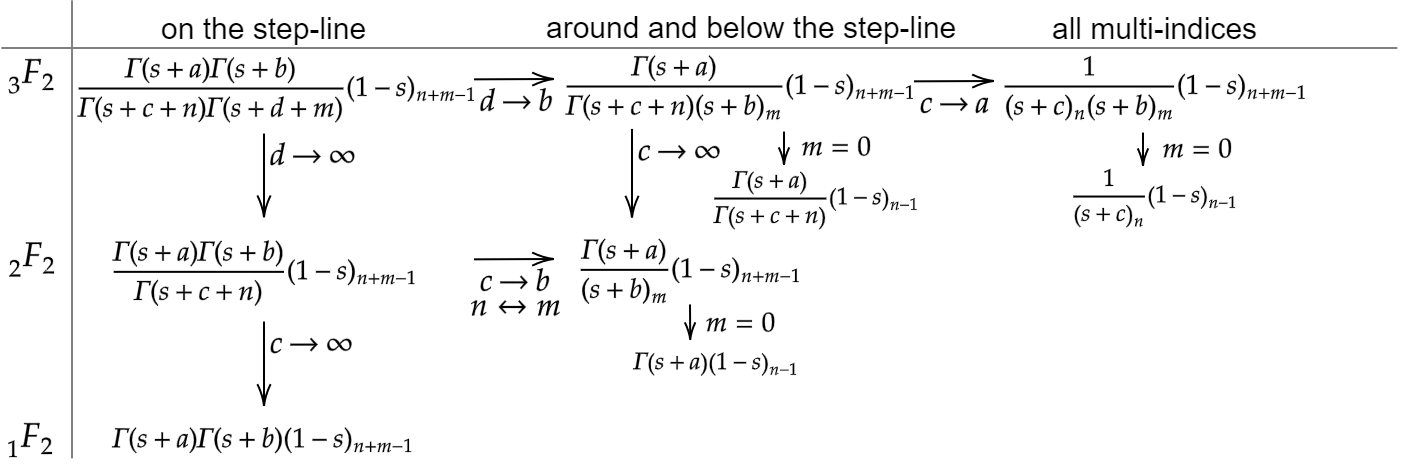}
    \caption{Underlying Mellin transforms.}
    \label{fig2}
\end{figure}
\medbreak

By taking suitable limits, it is also possible to consider weights with fewer Pochhammer symbols in the denominator. The associated multiple orthogonal polynomials will appear in the study of the singular values of (mixed) products of Ginibre matrices and truncated unitary matrices. The multiple orthogonal polynomials associated with weights with no Pochhammer in the denominator were already investigated in \cite{KuijlaarsZhang}, where they appeared in the study of products of Ginibre matrices. Weight functions for which the moments are ratios of Pochhammer symbols also appear naturally in the work of Sokal \cite{Sokal}, which makes a connection between multiple orthogonal polynomials, production matrices and branched continued fractions.

We may extend the multiple orthogonal polynomials in this article to multiple orthogonal polynomials with respect to more weights by considering additional exponential integral weights $(w_1,\dots,w_r)$, with $w_j(x) = x^\alpha E_{\nu_j+1}(x)$, that have parameters $\nu_1=-1$ (so that $w_1(x)=x^\alpha e^{-x}$) and $\nu_i - \nu_j \notin \mathbb{Z}$ whenever $i \neq j$. This system of weights will again be a Nikishin system under some additional restrictions on the parameters.

\end{document}